\documentclass[a4paper,10pt]{article}
\usepackage[utf8]{inputenc}
\usepackage{amsmath,amssymb,amsthm}
\usepackage{graphicx,graphics,color,enumitem,booktabs}
\usepackage{fancyhdr}
\usepackage{titlesec}
\usepackage{url}
\usepackage{tikz}
\newtheorem{assumption}{Assumption}[section]
\newtheorem{proposition}{Proposition}[section]

\newtheorem{theorem}{Theorem}[section]
\newtheorem{lemma}{Lemma}[section]
\newtheorem{remark}{Remark}[section]
\newtheorem{definition}{Definition}[section]
\newtheorem{problem_b}{Problem}[section]

\newcommand{\BA}{\textnormal{\textbf{A}}}   % A norm
\newcommand{\BM}{M}%{\textnormal{\textbf{M}}}   % M norm
\newcommand{\m}{\boldsymbol{m}}
\newcommand{\M}{\mathcal{M}}                % the upper bound of the operator \mathcal{A}(\xi)
\newcommand{\pr}{\textnormal{Pr}}           % projection to the tangent space
\newcommand{\prh}{\textnormal{Pr}_h}        % projection to the discrete tangent space
\newcommand{\wein}{\mathcal{H}}             % extended Weingarten map

\newcommand{\aast}{a^{\ast}}
\newcommand{\Btensor}{\mathcal{B}}
\newcommand{\Dtensor}{\mathcal{D}}

\def \S    {\mathcal{S}}

\newcommand{\A}{\mathcal{A}}
\newcommand{\diff}{\frac{\d}{\d t}}
\newcommand{\Ga}{\Gamma}
\newcommand{\Gat}{\Gamma(t)}
\newcommand{\GT}{\mathcal{G}_T}

\newcommand{\nbg}{\nabla_{\Gamma}}
\newcommand{\nbgh}{\nabla_{\Gamma_h}}
\newcommand{\mat}{\partial^{\bullet}}
\def \P {\mathcal{P}_h}
\newcommand{\Pt}{\widetilde{\mathcal{P}}_h}

\newcommand{\co}{continuous}
\def \d {\mathrm{d}}
\newcommand{\disp}{\displaystyle}

\newcommand{\inv}{^{-1}}
\newcommand{\la}{\langle}

\newcommand{\N}{\mathbb{N}}
\newcommand{\nb}{\nabla}

\newcommand{\op}{operator}
\newcommand{\pa}{\partial}
\newcommand{\R}{\mathbb{R}}
\newcommand{\ra}{\rangle}
\newcommand{\resp}{respectively}
\newcommand{\spn}{\textnormal{span}}
\newcommand{\st}{such that}

\def \t {(t)}
\newcommand{\Th}{\mathcal{T}_h}
\def \to {\rightarrow}
\newcommand{\vphi}{\varphi}
\newcounter{quotecount}

\newcommand\Linfty[2][]{%
  \lVert #2\rVert_{%
    L^{\infty}%
    \ifx\\#1\\\else(#1)\fi%
  }%
}
\usepackage{xifthen}
\newcommand{\Linftyy}[2][]{%
  \lVert #2\rVert_{
    L^{\infty}%
    \ifthenelse{\isempty{#1}}{}{(#1)}
  }
}
\newcommand\Ltwo[2][]{%
  \lVert #2\rVert_{%
    L^{2}%
    \ifx\\#1\\\else(#1)\fi%
  }%
}
\newcommand\Lone[2][]{%
  \lVert #2\rVert_{%
    L^{1}%
    \ifx\\#1\\\else(#1)\fi%
  }%
}
\newcommand\abs[1]{%
  \lvert #1\rvert
}

\usepackage{xparse}
\DeclareDocumentCommand{\Lnorm}{%
  O{} O{2} m%
}{\lVert #3\rVert_{%
    L^{#2}%
    \ifx\\#1\\\else(#1)\fi%
  }%
}

\DeclareDocumentCommand{\Wnorm}{%
  O{} O{1,\infty} m%
}{\lVert #3\rVert_{%
    W^{#2}%
    \ifx\\#1\\\else(#1)\fi%
  }%
}

\DeclareDocumentCommand{\Hnorm}{%
  O{} O{1} m%
}{\lVert #3\rVert_{%
    H^{#2}%
    \ifx\\#1\\\else(#1)\fi%
  }%
}

\newcommand\surface{\Ga}

\makeatletter
\newcommand{\oset}[3][0ex]{%
  \mathrel{\mathop{#3}\limits^{
    \vbox to#1{\kern-2\ex@
    \hbox{$\scriptstyle#2$}\vss}}}}
\makeatother

\DeclareDocumentCommand{\HspaceO}{%
  O{1} m%
}{\oset[0ex]{\circ\hphantom{#1}}{H^{#1}}\!\! (#2)%
}

\DeclareDocumentCommand{\WspaceO}{%
  O{k,p} m%
}{\oset[0ex]{\circ\hphantom{#1}}{W^{#1}}\!\! (#2)%
}

\DeclareDocumentCommand{\Hspace}{%
  O{1} m%
}{H^{#1}(#2)%
}
\DeclareDocumentCommand{\Wspace}{%
  O{k,p} m%
}{W^{#1}(#2)%
}

\makeatletter
\def\@seccntformat#1{\@ifundefined{#1@cntformat}%
   {\csname the#1\endcsname\quad}  % default
   {\csname #1@cntformat\endcsname}% enable individual control
}
\let\oldappendix\appendix %% save current definition of \appendix
\renewcommand\appendix{%
    \oldappendix
    \newcommand{\section@cntformat}{\appendixname~\thesection:\quad}
}
\makeatother
\begin{document}
\title{Error analysis for full discretizations of quasilinear parabolic problems on evolving surfaces}
\author{Bal\'{a}zs Kov\'{a}cs\footnote{MTA-ELTE NumNet Research Group, P{\'a}zm{\'a}ny P. s{\'e}t{\'a}ny 1/C, 1117 Budapest, Hungary; E-mail address: koboaet@cs.elte.hu \newline Present address: Mathematisches Institut, University of T\"{u}bingen} \, and 
Christian Andreas Power Guerra\footnote{Mathematisches Institut, University of T\"{u}bingen, Auf der Morgenstelle 10, 72076 T\"{u}bingen, Germany, E-mail address: power@na.uni-tuebingen.de}}
%\date{}
%\author{Bal\'{a}zs Kov\'{a}cs\footnote{blabla} \, and \, Christian Andreas Power Guerra\footnote{bla}}

\maketitle

\begin{abstract}
  Convergence results are shown for full discretizations of quasilinear parabolic partial differential equations on evolving surfaces. As a semidiscretization in space the evolving surface finite element method is considered, using a regularity result of a generalized Ritz map, optimal order error estimates for the spatial discretization is shown. Combining this with the stability results for Runge--Kutta and BDF time integrators, we obtain convergence results for the fully discrete problems.
\end{abstract}


\section{Introduction}
In this paper we show convergence of full discretizations of quasilinear parabolic partial differential equations on evolving surfaces. As a spatial discretization we consider the evolving surface finite element method. The resulting system of ordinary differential equations is discretized, either with an algebraically stable Runge--Kutta method, or with an implicit or linearly implicit backward differentiation formulae.

To our knowledge \cite{Elliott2015} is the only work on error analysis for nonlinear problems on evolving surfaces. They give semidiscrete error bounds for the Cahn--Hilliard equation. The authors are not aware of fully discrete error estimates published in the literature.

\medskip
We show convergence results for full discretizations of quasilinear parabolic problems on evolving surfaces with prescribed velocity. We prove unconditional stability and higher-order convergence results for Runge--Kutta and BDF methods. We show convergence as a full discretization when coupled with the ESFEM method as a space discretization for quasilinear problems. Similarly to the linear case the stability analysis is relying on energy estimates and multiplier techniques.

First, we generalize some geometric preturbation estimates to the quasilinear setting. We define a \emph{generalized Ritz map} for quasilinear operators, and use it to show optimal order error estimates for the spatial discretization. During the optimal order $L^2$-error bounds of the Ritz map we will use a similar argument as Wheeler in \cite{Wheeler_nonlinRitz}, %DoubleCITE
and elliptic regularity for evolving surfaces.
A further important point of the analysis is the required \emph{regularity} of the generalized Ritz map.
This will be used together with the assumed Lipschitz-type estimate for the nonlinearity, analogously as in \cite{DouglasDupont,LubichOstermann_RK,AkrivisLubich_quasilinBDF}.

We show stability and convergence results for the case of stiffly accurate algebraically stable implicit Runge--Kutta methods (having the Radau IIA methods in mind), and for an implicit and linearly implicit  $k$-step backward differentiation formulae up to order five. These results are relying on the techniques used in \cite{LubichOstermann_RK,DziukLubichMansour_rksurf} and \cite{AkrivisLubich_quasilinBDF,LubichMansourVenkataraman_bdsurf}. By combining the results for the spatial semidiscretization with stability and convergence estimates we show high-order convergence bounds for the fully discrete approximation.

\smallskip
A starting point of the finite element approximation to (elliptic) surface partial differential equations is the paper of Dziuk \cite{Dziuk88}. %DoubleCITE
Various convergence results for space discretizations of linear parabolic problems using the evolving surface finite element method (ESFEM) were shown in \cite{DziukElliott_ESFEM,DziukElliott_L2}, a fully discrete scheme was analysed in \cite{DziukElliott_fulldiscr}. These results are surveyed in \cite{DziukElliott_acta}.

The convergence analysis of full discretizations with higher-order time integrators within the ESFEM setting for linear problems were shown: for algebraically stable Runge--Kutta methods in \cite{DziukLubichMansour_rksurf}; for backward differentiation formulae (BDF) in \cite{LubichMansourVenkataraman_bdsurf}. The ESFEM approach and convergence results were later extended to wave equations on evolving surfaces, see \cite{LubichMansour_wave}.

A unified presentation of ESFEM and time discretizations for parabolic problems and wave equations can be found in \cite{diss_Mansour}.

A great number of real-life phenomena are modeled by nonlinear parabolic problems on evolving surfaces. Apart from general quasilinear problems on moving surfaces, see e.g.\ Example~3.5 in \cite{DziukElliott_SFEM}, more specific applications are the nonlinear models: diffusion induced grain boundary motion \cite{Cahn1997phase,Fife,Handwerker,Deckelnick2001,ElliottStyles_ALEnumerics}; Allen--Cahn and Cahn--Hilliard equations on evolving surfaces \cite{Cahn1996cahn,Elliott1996cahn,Elliott2015,Elliott2010a,Chen2002phase}; modeling solid tumor growth \cite{CGG,ElliottStyles_ALEnumerics}; pattern formation modeled by reaction-diffusion equations \cite{Leung2002,Madzvamuse2014exhibiting}; image processing \cite{Jin2004}; Ginzburg--Landau model for superconductivity \cite{Du2004}.

A number of nonlinear problems, in a general setting, were collected by Dziuk and Elliott in \cite{DziukElliott_ESFEM,DziukElliott_SFEM,DziukElliott_acta}, also see the references therein. A great number of nonlinear problems with numerical experiments were presented in the literature, see e.g.\ the above references, in particular \cite{DziukElliott_ESFEM,DziukElliott_SFEM,ElliottStyles_ALEnumerics,Deckelnick2001}.

\smallskip
The paper is organized in the following way:
In Section~\ref{section: problem} we formulate our problem and detail our assumptions.
In Section~\ref{section: ESFEM} we recall the evolving surface finite element method, together with some of its important properties and estimates. We introduce the generalized Ritz map, and show optimal order error estimates for the residual, using the crucial $W^{1,\infty}$ regularity estimate mentioned above.
Section~\ref{section: stability} covers the stability results and error estimates for Runge--Kutta and for implicit and linearly implicit BDF methods.
Section~\ref{section: error bounds} is devoted to the error bounds of the semidiscrete residual, which then leads to error estimates for the fully discretized problem.
In Section~\ref{section: extensions} we briefly discuss how our results can be extended to semilinear problems, and to the case where the upper and lower bounds of the elliptic part are depending on the norm of the solution.
Numerical results are presented in Section~\ref{section: numerics} to illustrate our theoretical results.

\section{The problem and assumptions}
\label{section: problem}

Let us consider a sufficiently smooth evolving closed hypersurface $\Ga\t \subset \R^{m+1}$ ($m\leq 2$), $0 \leq t \leq T$, which moves with a given smooth velocity $v$. Let $\mat u = \pa_{t} u + v \cdot \nb u$ denote the material derivative of the function $u$, where $\nbg$ is the tangential gradient given by $\nbg u = \nb u -\nb u \cdot \nu \nu$, with unit normal $\nu$. We are sharing the setting of \cite{DziukElliott_ESFEM,DziukElliott_L2}.

%We consider the following quasilinear problem:
%\begin{equation}%\label{eq: strong form}
%  \begin{cases}
%    \begin{alignedat}{3}
%        \mat u(x,t) + u(x,t) \nb_{\Gat} \cdot v(x,t) - \nb_{\Gat} \cdot \Big( \A(u(x,t)) \nb_{\Gat} u(x,t)\Big) &= f(x,t) \qquad & \textrm{ on } \Ga\t ,\\
%        u(x,0) &= u_0(x) \qquad & \textrm{ on } \Ga(0).
%    \end{alignedat}
%  \end{cases}
%\end{equation}
%\hfill[kb: choose from these two!]
%
We consider the following quasilinear problem for $u=u(x,t)$:
\begin{equation}
\label{eq: strong form}
    \begin{cases}
    \begin{alignedat}{4}
        \mat u + u \nb_{\Gat} \cdot v - \nb_{\Gat} \cdot \Big( \A(u) \nb_{\Gat} u\Big) &= f & \qquad & \textrm{ on } \Ga\t ,\\
        u(.,0) &= u_0 & \qquad & \textrm{ on } \Ga(0),
    \end{alignedat}
    \end{cases}
\end{equation}
where $\A:\R \to \R$ is sufficiently smooth function.

\begin{remark}
    The results of the paper can be generalized to the case of a sufficiently smooth matrix valued diffusion coefficient $\A(x,t,u) : T_x\Gat \to T_x\Gat$. The proofs are similar to the ones presented here, except they are more technical and lengthy, therefore they are not presented here.
\end{remark}

\medskip
The abstract setting of this quasilinear evolving surface PDE is a suitable combination of \cite[Section~1]{LubichOstermann_RK} and \cite[Section~2.3]{AlphonseElliottStinner}: Let $H\t$ and $V\t$ be real and separable Hilbert spaces (with norms $\|.\|_{H\t}$, $\|.\|_{V\t}$, \resp) such that $V\t$ is densely and \co ly embedded into $H\t$, and the norm of the dual space of $V\t$ is denoted by $\|.\|_{V\t'}$. The dual space of $H\t$ is identified with itself, and the duality $\la.,.\ra_t$ between $V\t'$ and $V\t$ coincides on $H\t\times V\t$ with the scalar product of $H\t$, for all $t\in[0,T]$.

The problem casts the following nonlinear operator:
\begin{equation*}
    \la A(u)v,w \ra_t =\int_{\Ga\t} \A(u)\nbg v \cdot \nbg w .
\end{equation*}

We assume that $A$ satisfies the following three conditions:

\noindent The bilinear form associated to the \op\ $A(u):V\t\to V\t'$ is \emph{elliptic} with $\m >0$
\begin{equation}\label{ellipticity}
    \la A(u)w,w \ra_t \geq \m \|w\|_{V\t}^2 \qquad (w \in V\t),
\end{equation}
uniformly in $u \in V\t$ and for all $t\in[0,T]$. It is \emph{bounded} with $\M >0$
\begin{equation}\label{boundedness}
    \big|\la A(u)v,w \ra_t\big| \leq \M \|v\|_{V\t} \|w\|_{V\t} \qquad (v,w \in V\t),
\end{equation}
uniformly in $u \in V\t$ and for all $t\in[0,T]$. We further assume that there is a subset $\S\t\subset V\t$ \st\ the following \emph{Lipschitz--type} estimate holds: for every $\delta>0$ there exists $L=L(\delta,(\S\t)_{0\leq t \leq T})$ \st\
\begin{equation}\label{Lipschitz}
    \big\| \big(A(w_1) - A(w_2)\big) u \big\|_{V\t'} \leq \delta\|w_1 - w_2\|_{V\t} + L \|w_1 - w_2\|_{H\t},
\end{equation}
for $u \in \S(t), \ w_1,w_2 \in V\t, \ 0\leq t \leq T$.

The above conditions were also used to prove error estimates using energy techniques in \cite{LubichOstermann_RK,DouglasDupont}, or more recently in \cite{AkrivisLubich_quasilinBDF}.

\medskip
The weak formulation uses Sobolev spaces on surfaces: For a sufficiently smooth surface
$\Gamma$ we define
\[
    H^{1}(\Gamma) = \bigl\{ \eta \in L^{2}(\Gamma) \mid \nbg \eta \in
    L^{2}(\Gamma)^{m+1} \bigr\},
\]
and analogously $H^{k}(\Gamma)$ for $k\in \N$ and $W^{k,p}(\Gamma)$ for $k\in \N, p \in [1,\infty]$, cf.\ \cite[Section~2.1]{DziukElliott_ESFEM}.
Finally, $\GT= \cup_{t\in[0,T]} \Ga\t\times\{t\}$ denotes the space-time manifold.

The weak problem corresponding to \eqref{eq: strong form} can be formulated by choosing the setting: $V\t=H^1(\Ga\t)$ and $H\t=L^2(\Ga\t)$, and the operator:
\begin{equation*}
    \la A(u)v,w \ra_t =\int_{\Ga\t} \A(u)\nbg v \cdot \nbg w .
\end{equation*}
The coefficient function $\A:\R \to \R$ satisfies the following conditions.
\begin{assumption}\label{assumptions}
  \begin{enumerate}
    \item[(a)] It is bounded, and Lipschitz--bounded with constant $\ell$.
    \item[(b)] The function $\A(s) \geq \m > 0$ for arbitrary $s \in \R$.
  \end{enumerate}
\end{assumption}

Throughout the paper we use the following subspace of $V\t$:
\begin{equation*}
    \S(t) := \S(t,r) = \big\{ u \in H^2(\Gat) \ \big| \, \|u\|_{W^{2,\infty}(\Gat)} \leq r \big\}.
\end{equation*}

Then the following proposition easily follows.
\begin{proposition}\label{prop: main properties}
    Under Assumption \ref{assumptions} and $u \in \S\t$ ($0\leq t \leq T$) the above \op\ $A$ satisfies the conditions \eqref{ellipticity}, \eqref{boundedness} and \eqref{Lipschitz} (with $\delta=0$).
\end{proposition}
\begin{proof}
The first two conditions \eqref{ellipticity} and \eqref{boundedness} are following from (a) and (b). Condition \eqref{Lipschitz} holds, since for $u\in\S\t$,\ $w_1, w_2 \in H^1(\Gat)$ and any $z \in H^1(\Gat)$, we have
\begin{alignat*}{2}
    \big|\big\la ( A(w_1)-A(w_2) )u,z \big\ra_t\big| &= \bigg|\int_{\Ga\t} \Big(\A(w_1) - \A(w_2)\Big) \nbg u \cdot \nbg z \bigg| \\
    %& \leq c\ell \int_{\Ga\t} \big| w_1 - w_2 \big| \big|\nbg u\big| \big|\nbg z\big| \\
    & \leq c\ell \ \|w_1-w_2\|_{L^2(\Gat)} \ r \ \|z\|_{H^1(\Gat)},
\end{alignat*}
where the constant $\ell$ is from Assumption \ref{assumptions} (a).
\end{proof}

\begin{definition}[Weak form]
    A function $u\in H^1(\GT)$ is called a \emph{weak solution} of \eqref{eq: strong form}, if for almost every $t\in[0,T]$
    \begin{equation}\label{eq: weak form}
        \diff \int_{\Gat}\!\!\!\! u \vphi + \int_{\Ga\t}\!\!\! \A(u) \nbg u \cdot \nbg \vphi = \int_{\Gat}\!\!\!\! u\mat \vphi
    \end{equation}
    holds for every  $\vphi \in H^1(\GT)$ and $u(.,0)=u_0$.% (where $\GT=\cup_{t\in[0,T]}\Gat\times\{t\}$).
\end{definition}

\section{Spatial semidicretization: evolving surface finite elements}
\label{section: ESFEM}

As a spatial semidiscretization we use the evolving surface finite element method introduced by Dziuk and Elliott in \cite{DziukElliott_ESFEM}. %DoubleCITE
We shortly recall some basic notations and definitions from \cite{DziukElliott_ESFEM}, for more details the reader is referred to Dziuk and Elliott \cite{Dziuk88,DziukElliott_L2,DziukElliott_acta}.%DoubleCITE

\subsection{Basic notations}
The smooth surface $\Gat$ is approximated by a triangulated one denoted by $\Ga_h(t)$, whose vertives are sitting on the surface, given as
\begin{equation*}
    \Ga_h(t) = \bigcup_{E(t)\in \Th(t)} E(t).
\end{equation*}
We always assume that the (evolving) simplices $E(t)$ are forming an admissible
triangulation $\Th(t)$, with $h$ denoting the maximum diameter.  Admissible triangulations were introduced in \cite[Section~5.1]{DziukElliott_ESFEM}: every $E(t)\in \Th(t)$ satisfies that the inner radius $\sigma_{h}$ is bounded from below by $ch$ with $c>0$, and $\Ga_{h}(t)$ is not a global double covering of $\Ga(t)$. Then the discrete tangential gradient on the discrete surface $\Ga_h\t$ is given by
\begin{equation*}
    \nb_{\Ga_h\t} \phi := \nb {\phi} - \nb {\phi} \cdot \nu_h \nu_h,
\end{equation*}
understood in a piecewise sense, with $\nu_h$ denoting the normal to $\Ga_h(t)$ (see \cite{DziukElliott_ESFEM}).

For every $t\in[0,T]$ we define the finite element subspace $S_h\t$ spanned by the continuous, piecewise linear evolving basis functions $\chi_j$, satisfying $\chi_j(a_i\t,t) = \delta_{ij}$ for all $i,j = 1, 2, \dotsc, N$, therefore
\begin{equation*}
    S_h\t = \spn\big\{ \chi_1( \, . \,,t), \chi_2( \, . \,,t), \dotsc, \chi_N( \, . \,,t) \big\}.
\end{equation*}

\noindent We interpolate the surface velocity on the discrete surface using the basis functions and denote it with $V_h$. Then the discrete material derivative is given by
\begin{equation*}
    \mat_h \phi_h = \pa_t \phi_h + V_h \cdot \nb \phi_h  \qquad (\phi_h \in S_h\t).
\end{equation*}
The key \textit{transport property} derived in \cite[Proposition 5.4]{DziukElliott_ESFEM}, is the following
\begin{equation}\label{eq: transport property}
    \mat_h \chi_k = 0 \qquad \textrm{for} \quad k=1,2,\dotsc,N.
\end{equation}

The spatially discrete quasilinear problem for evolving surfaces is formulated in
\begin{problem_b}[Semidiscretization in space]
    Find $U_h\in S_h\t$ \st\
    \begin{equation}\label{eq: semidiscrete problem}
            \diff \!\!\int_{\Ga_h\t}\!\! U_h \phi_h
            + \int_{\Ga_h\t}\!\!\! \A(U_h) \nbgh U_h \cdot \nbgh \phi_h = \int_{\Ga_h\t}\!\! U_h \mat_h \phi_h,  \qquad (\forall \phi_h \in S_h\t),
    \end{equation}
    with the initial condition $U_h( \, . \,,0)=U_h^0\in S_h(0)$ being a sufficient approximation to $u_0$.
\end{problem_b}

\subsection{The ODE system}
The ODE form of the above problem can be derived by setting
\begin{equation*}
    U_h( \, . \,,t) = \sum_{j=1}^N \alpha_j\t \chi_j( \, . \,,t)
\end{equation*}
into \eqref{eq: semidiscrete problem}, testing with $\phi_h=\chi_j$ and using the transport property \eqref{eq: transport property}.

\begin{proposition}[quasilinear ODE system]\label{prop: ODE system}
    The spatially semidiscrete problem \eqref{eq: semidiscrete problem} is equivalent to the following nonlinear ODE system for the vector $\alpha\t=(\alpha_j\t)\in\R^N$, collecting the nodal values of $U_h(.,t)$:
    \begin{equation}\label{eq: ODE system}
        \disp
        \begin{cases}
            \begin{alignedat}{2}
                \disp\diff \Big(M\t \alpha\t\Big) + A(\alpha\t) \alpha\t &= 0 \\
                \disp \alpha(0) &= \alpha_0
            \end{alignedat}
        \end{cases}
    \end{equation}
    where the evolving mass matrix $M\t$ and a nonlinear stiffness matrix $A(\alpha\t)$ are defined as
    \begin{equation*}
        M(t)_{kj} = \int_{\Ga_h\t}\!\!\!\! \chi_j \chi_k, \qquad A(\alpha\t)_{kj} =
        \int_{\Ga_h\t}\!\!\! \A\big(U_h\big) \nbgh \chi_j \cdot \nbgh  \chi_k,
%        \sum_{i,n=1}^{m+1}\int_{\Ga_h\t}\!\!\! a_{in} \big(U_h\big) \big(\nbgh \chi_j\big)_i \big(\nbgh  \chi_k \big)_n,
    \end{equation*}
    for $\alpha\t$ defining $U_h = \sum_{j=1}^N \alpha_j\t \chi_j(.,t)$.
\end{proposition}
The proof of this proposition is analogous to the corresponding one in \cite{DziukLubichMansour_rksurf}.

\subsection{Discrete Sobolev norm estimates} % Properties of evolving matrices
Through the paper we will work with the norm and semi-norm introduced in \cite{DziukLubichMansour_rksurf}. We denote these discrete Sobolev-type norms as
\begin{alignat*}{2}
    \disp |z\t|_{\BM\t} := \|Z_h\|_{L^2(\Ga_h\t)}, \qquad
    \disp |z\t|_{\BA\t} := \|\nbgh Z_h\|_{L^2(\Ga_h\t)},
\end{alignat*}

for arbitrary $z\t\in \R^N$, where $Z_h( \, . \,,t)=\sum_{j=1}^N z_j\t \chi_j( \, . \,,t)$, further by $\BM\t$ we mean the above mass matrix and by $\BA\t$ we mean the linear (but time dependent) stiffness matrix:
\begin{equation*}
    \BA\t_{kj} = \int_{\Ga_h\t}\!\!\! \nbgh \chi_j \cdot \nbgh  \chi_k.
\end{equation*}

A very important lemma in our analysis is the following:
\begin{lemma}[\cite{DziukLubichMansour_rksurf} Lemma 4.1]% and \cite{LubichMansourVenkataraman_bdsurf} Lemma 2.2]
\label{lemma: discrete sobolev norm estimates}
    There are constants $\mu, \kappa$ (independent of $h$) such that
    \begin{align*}
        %\label{eq: norm lemma M} %\int_{\Ga_h(s)}Z_hY_h - \int_{\Ga_h\t}Z_hY_h
        z^T \big( \BM(s) - \BM\t\big) y \leq&\ (e^{\mu(s-t)}-1) |z|_{\BM\t}|y|_{\BM\t}, \\
        %\label{eq: norm lemma Minv}
        %z^T \big( \BM\inv(s) - \BM\inv\t\big) y \leq&\ (e^{\mu(s-t)}-1) |z|_{M\inv\t}|y|_{M\inv\t}, \\
        %\label{eq: norm lemma A} %\int_{\Ga_h(s)}\!\!\! \nbgh Z_h\cdot\nbgh Y_h - \int_{\Ga_h\t}\!\!\! \nbgh Z_h\cdot\nbgh Y_h =:
        z^T \big( \BA(s) - \BA\t\big) y \leq&\ (e^{\kappa(s-t)}-1) |z|_{\BA\t}|y|_{\BA\t}
    \end{align*}
    for all $y,z\in \R^N$ and $s,t \in [0,T]$.
\end{lemma}
%We will use this lemma $s$ close to $t$, and then $(e^{\mu(s-t)}-1)\leq 2 \mu (s-t)$, in particular for $y=z$ we have
%\begin{align}
%  \label{eq: BM norm estimate}  |z|_{\BM(s)}^2 \leq (1+2\mu   (t-s)) |z|_{\BM\t}^2, \\
%  \label{eq: BA norm estimate} |z|_{\BA(s)}^2 \leq (1+2\kappa(t-s)) |z|_{\BA\t}^2.
%\end{align}

\subsection{Lifting process and approximation results}
\label{subsection: lift}

In the following we recall the so called \emph{lift operator}, which was introduced in \cite{Dziuk88} and further investigated in \cite{DziukElliott_ESFEM,DziukElliott_L2}. %The lift operator can be interpreted as a geometric projection: it maps a finite element function $\eta_{h}\colon \Ga_{h}\t \to \R$ on the discrete surface $\Gamma_{h}(t)$ onto a function $\eta_{h}^{l}\colon \Ga\t \to \R$ on the smooth surface $\Ga\t$, therefore it is crucial for our error estimates.
The lift operator projects a finite element function on the discrete surface onto a
function on the smooth surface.

Using the \emph{oriented distance function} $d$ (\cite[Section~2.1]{DziukElliott_ESFEM}), for a \co\ function $\eta_h \colon \Ga_h\t \to \R$ its lift is define as
\begin{equation*}
    \eta_{h}^{l}(p,t) := \eta_h(x,t), \qquad x\in\Ga\t,
\end{equation*}
where for every $x\in \Ga_{h}\t$ the value $p=p(x,t)\in\Gat$ is uniquely defined via $x = p + \nu(p,t) d(x,t)$. By $\eta^{-l}$ we mean the function whose lift is $\eta$.

We now recall some notions using the lifting process from \cite{Dziuk88,DziukElliott_ESFEM} and \cite{diss_Mansour}. We have the lifted finite element space
\begin{equation*}
    S_h^l\t := \big\{ \vphi_h = \phi_h^l \, | \, \phi_h\in S_h\t \big\}.
\end{equation*}
By $\delta_h$ we denote the quotient between the \co\ and discrete surface measures, $\d A$ and $\d A_h$, defined as $\delta_h \d A_h = \d A$. Further, we recall that
\begin{equation*}
    \pr := \big(\delta_{ij} - \nu_{i}\nu_{j}\big)_{i,j=1}^{m+1} \quad \textrm{and} \quad \prh := \big(\delta_{ij} - \nu_{h,i}\nu_{h,j}\big)_{i,j=1}^{m+1}
\end{equation*}
are the projections onto the tangent spaces of $\Ga$ and $\Ga_h$. Further, from \cite{DziukElliott_L2}, we recall the notation
\begin{equation*}
    Q_h = \frac{1}{\delta_h} (I-d\wein) \pr \prh \pr (I-d\wein),
\end{equation*}
where $\wein$ ($\wein_{ij} = \pa_{x_j}\nu_i$) is the (extended) Weingarten map. For these quantities we recall some results from \cite[Lemma 5.1]{DziukElliott_ESFEM}, \cite[Lemma 5.4]{DziukElliott_L2} and \cite[Lemma 6.1]{diss_Mansour}.
\begin{lemma}\label{lemma: geometric est}
    Assume that $\Ga_h\t$ and $\Ga\t$ is from the above setting, then we have the estimates:
    \begin{gather*}
        \|d\|_{L^\infty(\Ga_h\t)} \leq c h^2, \quad \|\nu_j\|_{L^\infty(\Ga_h\t)} \leq c h, \quad \|1-\delta_h\|_{L^\infty(\Ga_h\t)} \leq c h^2, \\
        \|\mat_h d \|_{L^\infty(\Ga_h\t)} \leq c h, \quad \|\pr - Q_h\|_{L^\infty(\Ga_h\t)} \leq c h^2, \quad \|\pr(\mat_hQ_h)\pr\|_{L^\infty(\Ga_h\t)} \leq c h^2 ,
    \end{gather*}
    with constants depending on $\GT$, but not on $t$.
\end{lemma}

\begin{lemma}
\label{lemma: Linfty aqui estimate}
    For $1\leq p\leq \infty$ there exists constants $c_{1},c_{2}>0$ independent of
    $t$ and $h$ such that the for all $u_{h}\in
    W^{1,p}\bigl(\surface_{h}(t)\bigr)$ it holds that $u_{h}^{l}\in
    W^{1,p}\bigl(\surface(t)\bigr)$ with the estimates
    \[
    c_{1} \Wnorm[\surface_{h}(t)][1,p]{u_{h}} \leq \Wnorm[\surface(t)][1,p]{u_{h}^{l}}
    \leq c_{2} \Wnorm[\surface_{h}(t)][1,p]{u_{h}}.
    \]
\end{lemma}
\begin{proof}
  The proofs follows easily from the relation $\nbgh u_{h} = \prh (I - d \wein) \nbg u_{h}^{l}$, cf.\ \cite[Lemma~3]{Dziuk88}.
\end{proof}

\subsection{Bilinear forms and their estimates}
\label{subsection: bilinear forms}

Apart from the $\xi$ dependence, we use the time dependent bilinear forms defined in \cite{DziukElliott_L2}: for arbitrary $z,\vphi, \xi \in H^1(\Ga)$, $\xi \in \S\t$, and their discrete analogs for $Z_h, \phi_h, \xi_h \in S_h$:
\begin{equation*}
    \begin{aligned}[c]
        m(z,\vphi)                &= \int_{\Ga\t}\!\!\!\! z \vphi, \\
        a(\xi;z,\vphi)            %&= \int_{\Ga\t}\! \sum_{i,j=1}^{m+1} a_{ij}(\xi) \big(\nbg z\big)_i \big(\nbg \vphi\big)_j, \\
		                          &= \int_{\Ga\t}\!\!\!\! \A(\xi) \nbg z \cdot \nbg \vphi, \\
        g(v;z,\vphi)              &= \int_{\Ga\t}\!\!\!\! (\nbg \cdot v) z\vphi, \\
        b(\xi;v;z,\vphi)          &= \int_{\Ga\t}\!\!\!\! \Btensor(\xi;v) \nbg z \cdot \nbg \vphi,
    \end{aligned}
    \quad
    \begin{aligned}[c]
        m_h(Z_h,\phi_h)           &= \int_{\Ga_h\t}\!\!\!\! Z_h \phi_h\\
        a_h(\xi_h;Z_h,\phi_h)     %&= \int_{\Ga_h\t}\! \sum_{i,j=1}^{m+1} a_{ij}^{-l}(\xi_h) \big(\nbgh Z_h\big)_i \big(\nbgh \phi_h\big)_j, \\
		                          &= \int_{\Ga_h\t}\!\!\!\! \A(\xi_h) \nbgh Z_h \cdot \nbgh \phi_h, \\
        g_h(V_h;Z_h,\phi_h)       &= \int_{\Ga_h\t}\!\!\!\! (\nbgh \cdot V_h) Z_h \phi_h, \\
        b_h(\xi_h;V_h;Z_h,\phi_h) &= \int_{\Ga_h\t}\!\!\!\! \Btensor_h(\xi_h;V_h) \nbg Z_h \cdot \nbg \phi_h,
    \end{aligned}
\end{equation*}
where the discrete tangential gradients are understood in a piecewise sense, and with the tensors given as
\begin{alignat*}{3}
    \Btensor(\xi;v)_{ij} &= \mat(\A(\xi)) + \nbg \cdot v \A(\xi) - 2 \A(\xi) \Dtensor(v),\\
    \Btensor_h(\xi_h;V_h)_{ij} &= \mat_h(\A(\xi_h)) + \nbgh \cdot V_h \A(\xi_h) - 2 \A(\xi_h) \Dtensor_h(V_h),\\
    \intertext{with}
    \Dtensor(v)_{ij} &= \frac{1}{2} \big( (\nbg)_i v_j + (\nbg)_j v_i \big),,\\
    \Dtensor_h(V_h)_{ij} &=  \frac{1}{2} \big( (\nbgh)_i (V_h)_j + (\nbgh)_j (V_h)_i \big),
\end{alignat*}
for $i,j=1,2,\dotsc,m+1$. For more details see \cite[Lemma 2.1]{DziukElliott_L2} (and the references in the proof), or \cite[Lemma 5.2]{DziukElliott_acta}.

We will also use the transport lemma (note that $\mat_h z_h = \pa_t z_h + v_h \nbg z_h$ for a $z_h \in S_h^l\t$):
\begin{lemma}\label{lemma: transport prop}
    For arbitrary $\xi_h^l \in S_h^l\t$ and $z_h, \ \vphi_h, \ \mat_h z_h, \ \mat_h \vphi_h \in S_h^l\t$ we have:
    {\setlength\arraycolsep{.13889em}
    \begin{eqnarray*}
        \diff m(z_h,\vphi_h) &=& m(\mat_h z_h,\vphi_h) + m(z_h,\mat_h \vphi_h) + g(v_h;z_h,\vphi_h), \\
        \diff a(\xi_h^l;z_h,\vphi_h) &=& a(\xi_h^l;\mat_h z_h,\vphi_h) + a(\xi_h^l;z_h,\mat_h \vphi_h) + b(\xi_h^l;v_h;z_h,\vphi_h),
    \end{eqnarray*}}
    where $v_h$ velocity of the surface.
\end{lemma}
\begin{proof}
This lemma can be shown analogously as \cite[Lemma~4.2]{DziukElliott_L2}, therefore the proof is omitted.
\end{proof}

Versions of this lemma with continuous material derivatives, or discrete bilinear forms are also true.

The following estimates will play a crucial role in the proofs.
\begin{lemma}[Geometric perturbation errors]
\label{lemma: estimation of forms}
    For any $\xi\in\S\t$, and $Z_h, \phi_h \in S_h\t$ with corresponding lifts $z_h, \vphi_h \in S_h^l\t$ we have the following bounds
    {\setlength\arraycolsep{.13889em}
    \begin{eqnarray*}
        \disp \big| m(z_h,\vphi_h) - m_h(Z_h,\phi_h) \big| &\leq& c h^2 \|z_h\|_{L^2(\Ga\t)} \|\vphi_h\|_{L^2(\Ga\t)}, \\
        \disp \big| a(\xi; z_h,\vphi_h) - a_h(\xi^{-l}; Z_h,\phi_h) \big| &\leq& c h^2 \|\nbg z_h\|_{L^2(\Ga\t)} \|\nbg \vphi_h\|_{L^2(\Ga\t)}, \\
        \disp \big| g(v_h;z_h,\vphi_h) - g_h(V_h;Z_h,\phi_h) \big| &\leq& c h^2 \|z_h\|_{L^2(\Ga\t)} \|\vphi_h\|_{L^2(\Ga\t)}, \\
        \disp \big| b(\xi;v_h; z_h,\vphi_h) - b_h(\xi^{-l};V_h; Z_h,\phi_h) \big| &\leq& c h^2 \|\nbg z_h\|_{L^2(\Ga\t)} \|\nbg \vphi_h\|_{L^2(\Ga\t)}.
    \end{eqnarray*}}
\end{lemma}

\begin{proof}
The first estimate was proved in \cite[Lemma 5.5]{DziukElliott_L2}, while the third can be found in \cite[Lemma~7.5]{LubichMansour_wave}.

The proof of the second estimate is similar to the linear case found in
\cite[Lemma~4.7]{DziukElliott_acta}. Again using the notation from \cite{DziukElliott_acta}:
\begin{equation*}
    Q_h = \frac{1}{\delta_h} (I-d\wein) \pr \prh \pr (I-d\wein)
\end{equation*}
%on $\Ga_h\t$, where $\delta_h$ is the ratio between the measures over the \co\ and the discrete surface, i.e.\ $\delta_h \d A_h = \d A$. Using this we can write
we obtain
\begin{equation}
\label{eq: transformation matrix Q}
     \A(\xi^{-l}) \nbgh Z_h \cdot \nbgh \phi_h = \delta_h \A(\xi^{-l}) Q_h \nbg z_h(p,.) \cdot \nbg \vphi_h(p,.).
\end{equation}
Similarly as in \cite[Lemma 5.5]{DziukElliott_L2}, the boundedness
(Proposition \ref{prop: main properties}) and the geometric estimate $\|\pr - Q_h\|_{L^\infty(\Ga_h)} \leq c h^2$ provides the estimate

\begin{align*}
    &\ \big| a(\xi; z_h,\vphi_h) - a_h(\xi^{-l}; Z_h,\phi_h) \big| \\
    =&\ \Big| \int_{\Ga\t}\!\!\!\!  \A(\xi) \nbg z_h \cdot \nbg \vphi_h \d A - \int_{\Ga_h\t}\!\!\!\!  \A(\xi^{-l}) \nbgh Z_h \cdot \nbgh \phi_h \d A_h \Big| \\
    =&\ \Big| \int_{\Ga\t}\!\!\!\!  \A(\xi) \nbg z_h \cdot \nbg \vphi_h \d A - \int_{\Ga_h\t}\!\!\!\!  \delta_h \A(\xi^{-l})  Q_h \nbg z_h(p,.) \cdot \nbg \vphi_h(p,.) \d A_h \Big| \\
    =&\ \Big| \int_{\Ga\t}\!\!\!\!  \A(\xi) \big(\pr-Q_h\big) \nbg z_h \cdot \nbg \vphi_h \d A \Big| \\
%    \leq&\ \M \| \big(\pr-Q_h\big) \nbg z_h\|_{L^2(\Ga\t)} \|\nbg \vphi_h\|_{L^2(\Ga\t)} \\
    \leq&\ \M c h^2 \|\nbg z_h\|_{L^2(\Ga\t)} \|\nbg \vphi_h\|_{L^2(\Ga\t)}.
\end{align*}

To prove the fourth estimate we follow \cite{LubichMansour_wave}: starting with the equality
\begin{equation*}
    \diff \int_{\Ga_h\t} \A^{-l}(\xi^{-l}) \nbgh Z_h \cdot \nbgh\phi_h = \diff \int_{\Ga\t} \A(\xi) Q_h^l \nbg z_h \cdot \nbg \vphi_h
\end{equation*}
then the transport lemma (Lemma \ref{lemma: transport prop} above) yields
\begin{align*}
    &\ \int_{\Ga_h\t} \A(\xi^{-l}) \mat_h \nbgh Z_h \cdot \nbgh\phi_h +  \int_{\Ga_h\t} \A^{-l}(\xi^{-l}) \nbgh Z_h \cdot \mat_h \nbgh\phi_h\\
    & +
    \int_{\Ga_h\t} \Btensor_h(\xi^{-l};V_h) \nbgh Z_h \cdot \nbgh\phi_h\\
    = &\  \int_{\Ga\t} \A(\xi) Q_h^l  \mat_h  \nbg z_h \cdot \nbg \vphi_h +  \int_{\Ga\t} \A(\xi) Q_h^l  \nbg z_h \cdot \mat_h \nbg \vphi_h \\
    & + \int_{\Ga\t} \Btensor(\xi;v_h) Q_h^l \nbg z_h \cdot \nbg \vphi_h
    + \int_{\Ga\t} \mat_h (\A(\xi)Q_h^l) \nbg z_h \cdot \nbg \vphi_h.
\end{align*}

Therefore using that the lift of $\mat_h Z_h$ is $\mat_h z_h$, \eqref{eq: transformation matrix Q} and Lemma \ref{lemma: geometric est} provides
\begin{align*}
    &|b_h(\xi^{-l};V_h;Z_h,\phi_h) - b(\xi;v_h;Z_h,\phi_h)| \\
    =&\ \Big| \int_{\Ga\t} \mat_h(\A(\xi)Q_h^l) \nbg z_h \cdot \nbg \vphi_h\Big| + \Big|\int_{\Ga\t} \Btensor(\xi;v_h) \big(Q_h^l - I \big) \nbg z_h \cdot \nbg \vphi_h \Big|\\
    &\ \leq c h^2 \|\nbg z_h\|_{L^2(\Ga\t)} \|\nbg \vphi_h\|_{L^2(\Ga\t)},
\end{align*}
where the last estimates follow from Lemma \ref{lemma: geometric est}, similarly as in \cite[Theorem~7.5]{LubichMansour_wave}.
\end{proof}

\subsection{Interpolation estimates}
By $I_h\colon H^{1}\bigl(\Ga (t)\bigr) \to S_{h}^{l}(t)$ we denote the finite element interpolation operator, having the error estimate below.
\begin{lemma}\label{lemma: interpolation error}
    For $m \leq 3$, there exists a constant $c>0$ independent of $h$ and $t$ such that for $u\in H^{2}\bigl(\Ga(t)\bigr)$:
    \begin{align*}
        \lVert u - I_{h} u \rVert_{L^{2}(\Gat)} + h \lVert \nbg( u - I_{h} u) \rVert_{L^{2}(\Gat)} \leq&\ c h^{2} \|u\|_{H^{2}(\Gat)}.
    \intertext{Furthermore, if $u\in W^{2,\infty}(\Gat)$, it also satisfies}
        \lVert \nbg(u - I_{h} u) \rVert_{L^{\infty}(\Gat)} \leq&\ c h \|u\|_{W^{2,\infty}(\Gat)},
    \end{align*}
    % kb: we don't use the more concrete estimates
%    \begin{align*}
%        \lVert u - I_{h} u \rVert_{L^{2}(\Ga)} + h \lVert \nbg( u - I_{h} u) \rVert_{L^{2}(\Ga)} \leq&\ c h^{2} \bigl(\lVert \nbg^{2} u\rVert_{L^{2}(\Ga)} + h \lVert \nbg u \rVert_{L^{2}(\Ga)}\bigr).
%    \intertext{Furthermore, if $u\in W^{2,\infty}(\Gat)$, it also satisfies}
%        \lVert \nbg(u - I_{h} u) \rVert_{L^{\infty}(\Ga)} \leq&\ c h \bigl(\lVert \nbg^{2} u\rVert_{L^{\infty}(\Ga)} + h \lVert \nbg u \rVert_{L^{\infty}(\Ga)}\bigr),
%    \end{align*}
    where $c>0$ is also independent of $h$ and $t$.
\end{lemma}
\begin{proof}
    The first inequality was shown in \cite{Dziuk88}. The dimension restriction is especially discussed in \cite[Lemma~4.3]{DziukElliott_acta}.

    The analogue of the second estimate for a reference element were shown in \cite[Theorem~3.1]{StrangFix}. Then using standard estimates of the reference element technique we obtain the stated result, cf.\ \cite{BrennerScott}.
\end{proof}

\subsection{The Ritz map for nonlinear problems on evolving surfaces}%: Bounds and Error bounds}
\label{section: Ritz}

Ritz maps for quasilinear PDEs on stationary domains were investigated by Wheeler in \cite{Wheeler_nonlinRitz}. %DoubleCITE
We generalize this idea for the case of quasilinear evolving surface PDEs. We define a generalized Ritz map for quasilinear elliptic operators, for the linear case see \cite{LubichMansour_wave}.

By combining the above definitions we set the following.
\begin{definition}[Ritz map]
\label{def: Ritz}
    For a given $z\in H^1(\Gat)$ and a given function $\xi\colon \Gat \to \R$ there is a
    unique $\Pt z\in S_h\t$ \st\ for all $\phi_h\in S_h\t$, with the corresponding lift
    $\vphi_h=\phi_h^l$, we have
    \begin{equation}\label{eq: Ritz definition}
        \disp a_h^{\ast}(\xi^{-l} ; \Pt z,\phi_h) = a^\ast(\xi ; z,\vphi_h),
    \end{equation}
    where $a^{\ast}:=a+m$ and $a_h^{\ast}:=a_h+m_h$, to make the forms $a$ and $a_h$ positive
    definite. Then $\P z \in S_h^l\t$ is defined as the lift of $\Pt z$, i.e.\ $\P z = (\Pt z)^l$.
\end{definition}
\noindent We remind here that by $\xi^{-l}$ we mean a function (living on the discrete surface) whose lift is $\xi$.

The Galerkin orthogonality does not hold in this case, just up to a small defect:
\begin{lemma}[pseudo Galerkin orthogonality]
\label{lemma: galerkin orthogonality}
    For any given $\xi\in\S\t$ %$\xi\in W^{1,\infty}(\Gat)$
    there holds, that for every $z\in H^{1}(\Gat)$ and $\varphi_{h}\in S_{h}^{l}(t)$
    \begin{equation}\label{eq: Gal-Ortho}
        \abs{\aast(\xi;z-\P z,\varphi_{h}) } \leq c h^2 \| \P z \|_{H^1(\Gat)} \|\vphi_h\|_{H^1(\Gat)},
    \end{equation}
    where $c$ is independent of $\xi$,\ $h$ and $t$.
\end{lemma}
\begin{proof}
Using the definition of the Ritz map:
{\setlength\arraycolsep{.13889em}
\begin{eqnarray*}
    \abs{\aast(\xi;z-\P z, \vphi_h)} %&=& \abs{\aast(\xi;z, \vphi_h) - \aast(\xi;\P z, \vphi_h)} \\
    &=& \abs{\aast_h(\xi^{-l}; \Pt z,  \phi_h) - \aast(\xi;\P z, \vphi_h)}\\
    &\leq&  \M c h^2 \| \P z \|_{H^1(\Gat)} \|\vphi_h\|_{H^1(\Gat)},
\end{eqnarray*}}
where we used Lemma \ref{lemma: estimation of forms}.
\end{proof}

\subsubsection*{Error bounds for the Ritz map and for its material derivatives}

In this section we prove error estimates for the Ritz map \eqref{eq: Ritz definition} and also for its material derivatives, the analogous results for the linear case can be found in \cite[Section~6]{DziukElliott_L2}, \cite[Section~7]{diss_Mansour}. The $\xi$ independency of the estimates requires extra care, previous results, e.g.\ the ones cited above, or \cite[Section~8]{LubichMansour_wave}, are not applicable.

\begin{theorem}
  \label{thm: Ritz error}
  The error in the Ritz map satisfies the bound, for arbitrary $\xi\in\S\t$ and $0 \leq t \leq T$ and $h \leq h_0$ with sufficiently small $h_0$,
  \begin{equation*}
    \|z-\P z\|_{L^2(\Gat)} + h \|z-\P z\|_{H^1(\Gat)} \leq c h^2 \|z\|_{H^2(\Gat)}.
  \end{equation*}
  where the constant $c$ is independent of $\xi$, $h$ and $t$ (but depends
  on $\m$ and $\M$).
%  Furthermore it also satisfies
%  \[
%  \Hnorm{\P z} \leq \Hnorm{z} + c h \Hnorm[][2]{z}.
%  \]
\end{theorem}

\begin{proof}
(a) We first prove the gradient estimate.

Starting by the ellipticity of the form $a$ and the non-negativity of the form $m$, then using the estimate \eqref{eq: Gal-Ortho} we have:
{\setlength\arraycolsep{.13889em}
\begin{eqnarray*}
    \m \|z-\P z\|_{H^1(\Gat)}^2 &\leq& \aast(\xi;z-\P z, z-\P z) \\
    &=& \aast(\xi;z-\P z, z-I_h z) + \aast(\xi;z-\P z, I_h z-\P z) \\
    &\leq& \M \|z-\P z\|_{H^1(\Gat)} \|z-I_h z\|_{H^1(\Gat)} \\
        & & + c h^2 \| \P z \|_{H^1(\Gat)} \|I_h z-\P z\|_{H^1(\Gat)} \\
    &\leq& \M c h \|z-\P z\|_{H^1(\Gat)} \|z\|_{H^2(\Gat)} \\
        &+& c h^2 \Big( \! 2\|z-\P z\|_{H^1(\Gat)}^2 \! + \! \|z\|_{H^1(\Gat)}^2 \! + \! ch^2\|z\|_{H^2(\Gat)}^2 \! \Big),
\end{eqnarray*}}
using the interpolation error, and for the second term we used the estimate
\begin{align*}
    &\ \| \P z \|_{H^1(\Gat)} \|I_h z-\P z\|_{H^1(\Gat)} \\
    \leq &\ \Big(\| \P z -z\|_{H^1(\Gat)} + \|z\|_{H^1(\Gat)} \Big) \Big(\|I_h z-z\|_{H^1(\Gat)} + \|z-\P z\|_{H^1(\Gat)}\Big)\\
    \leq &\ 2\|z-\P z\|_{H^1(\Gat)}^2 + \|z\|_{H^1(\Gat)}^2 + ch^2\|z\|_{H^2(\Gat)}^2.
\end{align*}
Now using Young's and Cauchy--Schwarz inequality, and for sufficiently small (but $\xi$ independent) $h$ we have the gradient estimate
\begin{equation*}
    \|z-\P z\|_{H^1(\Gat)}^2 \leq \frac{1}{\m} \M c h^2 \|z\|^2_{H^2(\Gat)}.
\end{equation*}

(b) The $L^2$-estimate follows from the Aubin-Nitsche trick. Let us consider the problem
\begin{equation*}
    -\nbg \cdot \big(\A(\xi)\nbg w\big) + w = z-\P z \qquad \textrm{on}\quad  \Ga\t,
\end{equation*}
then by % the usual elliptic theory (see, e.g.\
% \cite[Section3.1]{DziukElliott_acta}, or \cite{Aubinbook,Evans_PDE})
elliptic theory, cf.\ Theorem~\ref{lemma:EllipticRegularity},
we have the estimate, for the solution $w\in H^2(\Gat)$
\begin{equation*}
    \|w\|_{H^2(\Gat)} \leq c \|z-\P z\|_{L^2(\Gat)},
\end{equation*}
where $c$ is independent of $t$ and $\xi$. By testing the elliptic weak problem with $z-\P z$ we have
\begin{align*}
    \|z-\P z\|^2_{L^2(\Gat)} =&\ \aast(\xi;z-\P z,w) \\
    =&\ \aast(\xi;z-\P z,w - I_h w) + \aast(\xi;z-\P z,I_h w) \\
    \leq&\ \M \|z-\P z\|_{H^1(\Gat)} \|w - I_h w\|_{H^1(\Gat)} \\
     &\ + ch^2 \| \P z \|_{H^1(\Gat)} \|I_h w\|_{H^1(\Gat)} .
\end{align*}
Then the estimates of the interpolation error and combination of the above results yields
\begin{equation*}
    \|z-\P z\|_{L^2(\Gat)} \frac{1}{c} \|w\|_{H^2(\Gat)} \leq \|z-\P z\|^2_{L^2(\Gat)} \leq \M c h^2 \|z\|_{H^2(\Gat)} \|w\|_{H^2(\Gat)},
\end{equation*}
which completes the proof of the first assertion.
\end{proof}

We will also need the following error estimates for the material derivatives of the Ritz map.
\begin{theorem}
\label{thm: Ritz mat error}
    The error in the material derivatives of the Ritz map satisfies the bounds, for $k \geq 1$, and for arbitrary $\xi\in\S\t$ and $0 \leq t \leq T$ and $h \leq h_0$ with sufficiently small $h_0$,
    \begin{align*}
        \|(\mat_h)^{(k)}(z-\P z)\|_{L^2(\Gat)} + h &\|\nbg(\mat_h)^{(k)}(z-\P z)\|_{L^2(\Gat)}  \leq \M c_{k} h^2 \sum_{j=1}^{k} \|(\mat_h)^{(j)}z\|_{H^2(\Gat)}.
    \end{align*}
    The constant $c_{k}>0$ is independent of $\xi$ and $h$ (but depends on $\alpha$ and $\M$).
\end{theorem}
\begin{proof}
The proof is a modification of \cite[Theorem 7.3]{diss_Mansour}.

\medskip
For $k=1$: (a) We start by taking the time derivative of the definition of the Ritz map \eqref{eq: Ritz definition}, use the transport properties (Lemma \ref{lemma: transport prop}), and use the definition of the Ritz map once more, we arrive at
{\setlength\arraycolsep{.13889em}
\begin{eqnarray*}
    \aast(\xi;\mat_h z,\vphi_h) &=& -b(\xi;v_h;z,\vphi_h) - g(v_h;z,\vphi_h) \\
    & & +\aast_h(\xi^{-l};\mat_h \Pt z,\phi_h) +b_h(\xi^{-l};V_h;\Pt z,\phi_h) +g_h(V_h;\Pt z,\phi_h).
\end{eqnarray*}}
Then we obtain
\begin{align}
    \aast(\xi;\mat_h z-\mat_h \P z ,\vphi_h) =&\ -b(\xi;v_h;z - \P z,\vphi_h) - g(v_h;z - \P z,\vphi_h) \nonumber \\
     &\ +F_1(\vphi_h), \label{eq: mat error ritz - main equation}
\end{align}
where
\begin{align*}
     F_1(\vphi_h) =&\ \big(\aast_h(\xi^{-l};\mat_h \Pt z,\phi_h) - \aast(\xi;\mat_h \P z ,\vphi_h)\big) \\
                  &\ + \big(b_h(\xi^{-l};V_h;\Pt z,\phi_h) - b(\xi;v_h;\P z,\vphi_h)\big) \\
                  &\ + \big(g_h(V_h;\Pt z,\phi_h) - g(v_h;\P z,\vphi_h)\big).
\end{align*}
Using the geometric estimates of Lemma \ref{lemma: estimation of forms} $F_1$ can be estimated as
\begin{equation*}
    \big|F_1(\vphi_h)\big| \leq c \M h^2 \big( \|\mat_h\P z\|_{H^1(\Gat)} + \|\P z\|_{H^1(\Gat)}\big) \|\vphi_h\|_{H^1(\Gat)}.
\end{equation*}
Then using $\mat_h \P z$ as a test function in \eqref{eq: mat error ritz - main equation}, and using the error estimates of the Ritz map, together with the estimates above, with $h\leq h_0$ independent of $\xi$, we have
\begin{equation*}
    \|\mat_h \P z\|_{H^1(\Gat)} \leq \M c\|\mat z\|_{H^1(\Gat)} + \M c h \|z\|_{H^2(\Gat)}.
\end{equation*}
Combining all the previous estimates and using Young's inequality, Cauchy--Schwarz inequality, for sufficiently small ($\xi$ independent) $h\leq h_0$, we obtain
\begin{equation*}
     \aast(\xi;\mat_h z-\mat_h \P z ,\vphi_h) \leq \M c h \Big( \|z\|_{H^2(\Gat)} + h\|\mat z\|_{H^1(\Gat)} \Big) \|\vphi_h\|_{H^1(\Gat)}.
\end{equation*}
Then as in the previous proof we have
\begin{align*}
    \m \|\mat_hz-\mat_h\P z\|_{H^1(\Gat)}^2 
    \leq&\ \aast(\xi;\mat_h z-\mat_h \P z, \mat_h z-\mat_h \P z) \\
    =&\ \aast(\xi;\mat_h z - \mat_h \P z, \mat_h z-I_h \mat z) + \aast(\xi;\mat_h z-\mat_h \P z, I_h \mat z-\mat_h \P z) \\
    \leq&\ \M \|\mat_h z-\mat_h \P z\|_{H^1(\Gat)} \|\mat_h z-I_h \mat z\|_{H^1(\Gat)} \\
         &\ + \M c h \Big( \|z\|_{H^2(\Gat)} + h\|\mat z\|_{H^1(\Gat)} \Big) \|I_h \mat z-\mat_h \P z\|_{H^1(\Gat)}.
\end{align*}
Then the interpolation estimates, Young's inequality, absorption using $h\leq h_0$, yields the gradient estimate.

(b) The $L^2$-estimate again follows from the Aubin-Nitsche trick. Let us now consider the problem
\begin{equation*}
    -\nbg \cdot \big(\A(\xi)\nbg w\big) + w = \mat_h z-\mat_h \P z \qquad \textrm{on}\quad  \Ga\t,
\end{equation*}
together with the elliptic estimate (cf.\ Theorem~\ref{lemma:EllipticRegularity}), for the solution $w\in H^2(\Gat)$
\begin{equation*}
    \|w\|_{H^2(\Gat)} \leq c\|\mat_h z-\mat_h \P z\|_{L^2(\Gat)},
\end{equation*}
again, $c$ is independent of $t$ and $\xi$.

Then a similar calculation as \cite[Theorem~6.2]{DziukElliott_L2}, \cite[Theorem~7.3]{diss_Mansour} provides the $L^2$-norm estimate.

For $k>1$ the proof is analogous.
\end{proof}

\subsubsection*{Regularity of the Ritz map}

The following technical result will play an important role in showing optimal bounds of the semidiscrete residual.

\begin{lemma}\label{lemma: Ritz regularity lemma}
  For $m \leq 2$, there exists a constant $c>0$ independent of $h$ and $t$ such that for a function $u \in W^{2,\infty}(\Gat)$  for all $t\in[0,T]$, the following estimate holds
  \[
    \Lnorm[\surface(t)][\infty]{ \nbg \P u } \leq c \Wnorm[\surface(t)][2,\infty]{u}.
  \]
\end{lemma}

\begin{proof}
Using the triangle inequality we start to estimate as
\begin{align*}
    \|\nbg \P u\|_{L^\infty(\Gat)} %\leq&\ \|\nbg (\P u - u)\|_{L^\infty(\Gat)} + \|\nbg u\|_{L^\infty(\Gat)} \\
    \leq&\ \|\nbg (\P u - I_hu)\|_{L^\infty(\Gat)} + \|\nbg (I_h u - u)\|_{L^\infty(\Gat)} \\
     &\ + \|\nbg u\|_{L^\infty(\Gat)}.
\end{align*}
The last term is harmless. The second term is estimated using Lemma~\ref{lemma: interpolation error}. For the first term, using the inverse estimate, error estimates for the Ritz map and for the interpolation operator we obtain
\begin{align*}
    \|\nbg (\P u - I_hu)\|_{L^\infty(\Gat)} \leq&\ ch^{-m/2} \|\nbg (\P u - I_hu)\|_{L^2(\Gat)} \\
    \leq&\ ch^{-m/2} \Big(\|\nbg (\P u - u)\|_{L^2(\Gat)} + \|\nbg (u - I_hu)\|_{L^2(\Gat)} \Big) \\
    \leq&\ ch^{-m/2}h \|u\|_{H^2(\Gat)} \leq c \Wnorm[\surface(t)][2,\infty]{u}.
\end{align*}
\end{proof}

\begin{remark}
  A stronger result holds, assuming that $u\in W^{1,\infty}(\surface(t))$, the bound $\Lnorm[\surface(t)][\infty]{\nbg \P u}\leq c \Wnorm[\surface(t)][1,\infty]{u}$ can be shown. However, the proof is technical and requires more sophisticated arguments, cf.\ \cite{diss_Power}. This enables to weaken the assumption to $W^{1,\infty}$ in the definition of the $\S\t$ set. We do not include these results here because of their length.
\end{remark}

\section{Time discretizations: stability}
\label{section: stability}

\subsection{Runge--Kutta methods}
We consider an $s$-stage algebraically stable implicit Runge--Kutta (R--K) method for the time discretization of the ODE system \eqref{eq: ODE system}, coming from the ESFEM space discretization of the quasilinear parabolic evolving surface PDE.

In the following we extend the stability result for R--K methods of \cite[Lemma~7.1]{DziukLubichMansour_rksurf}, to the case of quasilinear problems. Apart form the properties of the ESFEM the proof is based on the energy estimation techniques, see Lubich and Ostermann \cite[Theorem~1.1]{LubichOstermann_RK}. Generally on Runge--Kutta methods we refer to \cite{HairerWannerII}.

\medskip
For the convenience of the reader we recall the method: for simplicity, we assume equidistant time steps $t_{n}:= n\tau$, with step size $\tau$. Our results can be straightforwardly extended to the case of nonuniform time steps.  The $s$-stage implicit Runge--Kutta method, defined by the given Butcher tableau
\[
  \begin{array}{c|c}
    (c_{i}) & (a_{ij}) \\
     \hline & (b_{i})
  \end{array} \qquad \text{for}\quad i,j = 1,2,\dotsc, s,
  \]
  applied to the system \eqref{eq: ODE system}, reads as
\begin{subequations}
    \begin{alignat*}{3}
%    \label{eq_rk-a}
        M_{ni} \alpha_{ni} &= M_{n} \alpha_{n} + \tau \sum_{j=1}^{s} a_{ij}
        \dot{\alpha}_{nj}, \qquad &&\text{for} \quad i=1,2,\dotsc,s, \\
%    \label{eq_rk-b}
        M_{n+1} \alpha_{n+1} &= M_{n} \alpha_{n} + \tau \sum_{i=1}^{s} b_{i}
        \dot{\alpha}_{ni},&&
    \intertext{where the internal stages satisfy}
        \nonumber
        0 &= \dot{\alpha}_{ni} + A(\alpha_{ni}) \alpha_{ni} \qquad &&\text{for} \quad i=1,2,\dotsc,s,
    \end{alignat*}
\end{subequations}

\noindent with $M_{ni}:=M(t_{n}+c_{i}\tau)$ and $M_{n+1}:=M(t_{n+1})$. Here $\dot{\alpha}_{ni}$ is not a derivative but a suggestive notation. \par
We recall that the fully discrete solution is $U_{h}^{n} = \sum_{j = 1 }^{N}  \alpha_{n,j} \chi_{j}(\, . \, , t_{n}) $.

For the R--K method we make the following assumptions:
\begin{assumption}\label{assump: RK method assumptions}
    \begin{itemize}
        \item The method has stage order $q\geq 1$ and classical order $p\geq q+1$.
        \item The coefficient matrix $(a_{ij})$
            is invertible.%; its inverse will be denoted by upper indices $(a^{ij})$.
        \item The method is \emph{algebraically stable}, i.e.\ $b_{j}>0$ for $j=1,2,\dotsc,s$ and the following matrix is positive semi-definite:
            \begin{align*}
            %\label{eq: algebraic stability}
                \big(b_{i}a_{ij} - b_{j}a_{ji} -b_{i}b_{j}\big)_{i,j=1}^{s}.
            \end{align*}
        \item The method is \emph{stiffly accurate}, i.e.\ $b_{j}=a_{sj}$, and $c_{s}=1$ for $j=1,2,\dotsc,s$.
    \end{itemize}
\end{assumption}

Instead of \eqref{eq: ODE system}, let us consider the following perturbed version of the equation:
\begin{equation}\label{eq: pertubated ODE}
    \begin{cases}
        \begin{alignedat}{2}
            \disp\diff \big(M\t \widetilde{\alpha}\t\big) + A(\widetilde{\alpha}\t) \widetilde{\alpha}\t &= M\t r\t \\
            \disp \widetilde{\alpha}(0) &= \widetilde{\alpha}_0.
        \end{alignedat}
    \end{cases}
\end{equation}
The substitution of the true solution $\widetilde{\alpha}\t$ of the perturbed problem into the R--K method, yields the defects $\Delta_{ni}$ and $\delta_{ni}$, by setting $e_n = \alpha_n - \widetilde{\alpha}(t_n)$, $E_{ni} = \alpha_{ni} - \widetilde{\alpha}(t_n+c_i\tau)$ and $\dot{E}_{ni} = \dot{\alpha}_{ni} - \dot{\widetilde{\alpha}}(t_n+c_i\tau)$, then by subtraction the following \emph{error equations} hold:

\begin{subequations}
  \begin{align*}
    %\label{eq: error eq a}
    M_{ni}E_{ni}  &= M_{n}e_{n} + \tau\sum_{j=1}^{s} a_{ij} \dot{E}_{nj}  -
    \Delta_{ni}, \qquad \text{for} \quad i=1,2,\dotsc,s, \\
    %\label{eq: error eq b}
    M_{n+1} e_{n+1} &= M_{n} e_{n} + \tau \sum_{i=1}^{s} b_{i} \dot{E}_{ni}  -
    \delta_{n+1},
  \end{align*}
\end{subequations}
where the internal stages satisfy:

\begin{align*}
  %\label{eq: error eq interStages}
  \dot{E}_{ni} + A(\alpha_{ni}) E_{ni} = -\big( A(\alpha_{ni}) - A(\widetilde{\alpha}_{ni}) \big) \widetilde{\alpha}_{ni}- M_{ni} r_{ni}, \quad \text{for} \quad i=1,2,\dotsc,s,
\end{align*}
with $r_{ni} := r(t_{n} + c_{i}\tau)$.

Now we state one of the key lemmas of this paper, which provide unconditional stability for the above class of Runge--Kutta methods.
\begin{lemma}
\label{lemma: RK stability}
    For an $s$-stage implicit Runge--Kutta method satisfying Assumption \ref{assump: RK method assumptions}. If the equation \eqref{eq: weak form} has a solution in $\S\t$ for $0\leq t \leq T$. Then there exists a $\tau_0 >0$, \st\ for $\tau\leq\tau_0$ and $t_n=n \tau\leq T$, that the error $e_n$ is bounded by
    \begin{alignat*}{2}
        |e_n|_{\BM_n}^2 + \tau \sum_{k=1}^{n} |e_k|_{\BA_k}^2
        \leq&  C \bigg( |e_0|_{\BM_0}+\tau \sum_{k=1}^{n-1} \sum_{i=1}^{s} \|M_{ki} r_{ki}\|_{\ast, t_{ki}}^2
        + \tau \sum_{k=1}^{n} \Big|\frac{\delta_k}{\tau}\Big|_{\BM_k}^2 \\
        & + C \tau \sum_{k=0}^{n-1} \sum_{i=1}^{s} \Bigl( |M_{ki}^{-1}\Delta_{ki}|_{\BM_{ki}}^2
        + |M_{ki}^{-1}\Delta_{ki}|_{\BA_{ki}}^2\Bigr) \bigg),
    \end{alignat*}
    where $\|w\|_{\ast, t}^2=w^T(\BA\t+\BM\t)\inv w$. The constant $C$ is independent of $h, \ \tau$ and $n$ (but depends on $\m$, $\M$, $L$, $\mu$, $\kappa$ and $T$).
\end{lemma}
\begin{proof}
The combination of proofs of Theorem 1.1 from \cite{LubichOstermann_RK} and of Lemma 7.1 from \cite{DziukLubichMansour_rksurf} (or \cite[Lemma 3.1]{diss_Mansour}) suffices, therefore it is omitted here. To be precise, the proof of this result is more closely related to \cite{DziukLubichMansour_rksurf}, except the estimates involving the internal stages are more similar to \cite{LubichOstermann_RK}.
\end{proof}

Then, using the above stability results, the error bounds are following analogously as in \cite[Theorem 8.1]{DziukLubichMansour_rksurf} (or \cite[Theorem 5.1]{diss_Mansour}).
\begin{theorem}
\label{thm: RK error estimates}
    Consider the quasilinear parabolic problem \eqref{eq: strong form}, having a solution in $\S\t$ for $0\leq t \leq T$. Couple the evolving surface finite element method as space discretization with time discretization by an $s$-stage implicit Runge--Kutta method satisfying Assumption \ref{assump: RK method assumptions}. Assume that the Ritz map of the solution has \co\ discrete material derivatives up to order $q+2$. Then there exists $\tau_0>0$, independent of $h$, \st\ for $\tau \leq \tau_0$, for the error $E_h^n=U_h^n-\P u(.,t_n)$ the following estimate holds for $t_n=n\tau \leq T$:
    \begin{gather*}
        \|E_h^n\|_{L^2(\Ga_h(t_n))} + \Big( \tau \sum_{j=1}^n \|\nb_{\Ga_h(t_j)} E_h^j \|_{L^2(\Ga_h(t_j))}^2 \Big)^{\frac{1}{2}} \\
        \leq C \tilde{\beta}_{h,q} \tau^{q+1} \! + \! C \Big( \! \tau \!\! \sum_{k=0}^{n-1} \sum_{i=1}^s \|R_h(.,t_k\! +\! c_i\tau) \|_{H^{-1}(\Ga_h(t_k+c_i\tau))}^2 \! \Big)^{\frac{1}{2}} \!\! + C \|E_h^0\|_{L^2(\Ga_h(0))},
    \end{gather*}
    where the constant $C$ is independent of $h, \ \tau$ and $n$ (but depends on $\m$, $\M$, $L$, $\mu$, $\kappa$ and $T$). Furthermore
    \begin{align*}
        \tilde{\beta}_{h,q}^2 =& \int_0^T \sum_{\ell=1}^{q+2} \| (\mat_h)^{(\ell)} (\P u)(.,t) \|_{L^2(\Ga_h\t)} \d t \\
        + & \int_0^T \sum_{\ell=1}^{q+1} \| \nb_{\Ga_h\t} (\mat_h)^{(\ell)} (\P u)(.,t) \|_{L^2(\Ga_h\t)} \d t.
    \end{align*}
    The $H^{-1}$ norm  of $R_h$ is defined as
    \begin{equation*}
        \disp \|R_h(.,t) \|_{H^{-1}(\Ga_h\t)} := \sup_{0\neq\phi_h\in S_h\t} \frac{\la R_h(.,t),\phi_h\ra_{L^2(\Ga_h\t)}}{\|\phi_h\|_{H^{1}(\Ga_h\t)}} \ .
    \end{equation*}
\end{theorem}

\subsection{Backward differentiation formulae}

We apply a $k$-step backward difference formula (BDF) for $k\leq5$ as a discretization to the ODE system \eqref{eq: ODE system}, coming from the ESFEM space discretization of the quasilinear parabolic evolving surface PDE. Both implicit and linearly implicit methods are discussed.

In the following we extend the stability result for BDF methods of \cite[Lemma~4.1]{LubichMansourVenkataraman_bdsurf}, to the case quasilinear problems. Apart from the properties of the ESFEM the proof is based on Dahlquist's G--stability theory  \cite{Dahlquist} %DoubleCITE
and on the multiplier technique of Nevanlinna and Odeh \cite{NevanlinnaOdeh}.%DoubleCITE

\medskip
We recall the $k$-step BDF method for \eqref{eq: ODE system} with step size $\tau>0$:
\begin{equation}\label{def: BDF}
    \disp \frac{1}{\tau} \sum_{j=0}^k \delta_j M(t_{n-j})\alpha_{n-j} + A(\alpha_n)\alpha_n = 0, \qquad (n \geq k),
\end{equation}
where the coefficients of the method are given by $\delta(\zeta)=\sum_{j=0}^k \delta_j \zeta^j=\sum_{\ell=1}^k \frac{1}{\ell}(1-\zeta)^\ell$, while the starting values are $\alpha_0, \alpha_1, \dotsc, \alpha_{k-1}$. The method is known to be $0$-stable for $k\leq6$ and have order $k$ (for more details, see \cite[Chapter~V.]{HairerWannerII}).

Similarly linearly implicit method modification is, using the polynomial $\gamma(\zeta) = \sum_{j=1}^{k} \gamma_j \zeta^j = \zeta^k -(\zeta-1)^{k-1}$:
\begin{equation}\label{def: linearly implicit BDF}
    \frac{1}{\tau} \sum_{j=0}^k \delta_j M(t_{n-j})\alpha_{n-j} + A\Big(\sum_{j=1}^k\gamma_j \alpha_{n-j}\Big)\alpha_n = 0, \qquad (n \geq k).
\end{equation}
For more details we refer to \cite{AkrivisLubich_quasilinBDF}.

\medskip
Instead of \eqref{eq: ODE system} let us consider again the perturbed problem \eqref{eq: pertubated ODE}. By substituting the true solution $\widetilde{\alpha}\t$ of the perturbed problem into the BDF method \eqref{def: BDF}, we obtain
\begin{equation*}
    \frac{1}{\tau} \sum_{j=0}^k \delta_j M(t_{n-j})\widetilde{\alpha}_{n-j} + A(\widetilde{\alpha}_n)\widetilde{\alpha}_n = -d_n, \qquad (n \geq k).
\end{equation*}
By introducing the error $e_n = \alpha_n - \widetilde{\alpha}(t_n)$, multiplying by $\tau$, and by subtraction we have the error equation
\begin{equation*}%\label{eq: BDF error eq}
    \sum_{j=0}^k \delta_j M_{n-j} e_{n-j} + \tau A(\alpha_n) e_n + \tau \big( A(\alpha_n) - A(\widetilde{\alpha}_n) \big) \widetilde{\alpha}_n = \tau d_n, \quad (n \geq k).
\end{equation*}
In the linearly implicit case we obtain:
\begin{equation*}
    \sum_{j=0}^k \delta_j M_{n-j} e_{n-j} + \tau A\Big(\sum_{j=1}^k\gamma_j \alpha_{n-j}\Big) e_n + \tau \Big( A\Big(\sum_{j=1}^k\gamma_j \alpha_{n-j}\Big) - A\Big(\sum_{j=1}^k\gamma_j \widetilde{\alpha}_{n-j}\Big) \Big) \widetilde{\alpha}_n = \tau \hat{d}_n, \quad (n \geq k),
\end{equation*}
where $\hat{d}_n$ have similar properties as $d_n$, therefore it will be also denoted by $d_n$.

%We recall two important preliminary results.
%\begin{lemma}[G.~Dahlquist \cite{Dahlquist}]%DoubleCITE
%    Let $\delta(\zeta)$ and $\mu(\zeta)$ be polynomials of degree at most $k$ (at least one of them of exact degree $k$) that have no common divisor. Let $\la \,.\, | \,.\, \ra$ be an inner product on $\R^N$ with associated norm $\| \,.\, \|$. If
%    \begin{equation*}
%        \textnormal{Re} \frac{\delta(\zeta)}{\mu(\zeta)} > 0, \qquad \textrm{for} \quad |\zeta|<1,
%    \end{equation*}
%    then there exists a symmetric positive definite matrix $G = (g_{ij}) \in \R^{k\times k}$ and real $\gamma_0,\dotsc,\gamma_k$ such that for all $v_0,\dotsc,v_k\in\R^N$
%    \begin{equation*}
%        \Big\la \sum_{i=0}^k \delta_i v_{k-i} \Big| \sum_{i=0}^k \mu_i v_{k-i}  \Big\ra = \sum_{i,j=1}^k g_{ij} \la v_i \,|\, v_j \ra - \sum_{i,j=1}^k g_{ij} \la v_{i-1} \,|\, v_{j-1} \ra + \Big\| \sum_{i=0}^k \gamma_i v_i \Big\|^2
%      \end{equation*}
%      holds.
%\end{lemma}
%
%Together with this result, the case $\mu(\zeta)=1-\eta\zeta$ will play an important role:
%\begin{lemma}[O.~Nevanlinna \& F.~Odeh \cite{NevanlinnaOdeh}]%DoubleCITE
%    If $k\leq5$, then there exists $0\leq\eta<1$ \st\ for $\delta(\zeta)=\sum_{\ell=1}^k \frac{1}{\ell}(1-\zeta)^\ell$,
%    \begin{equation*}
%        \textnormal{Re} \frac{\delta(\zeta)}{1-\eta\zeta} > 0, \qquad \textrm{for} \quad |\zeta|<1.
%    \end{equation*}
%    The smallest possible values of $\eta$ is found to be $\eta=\, 0, \quad 0, \quad 0.0836, \quad 0.2878, \quad 0.8160$ for $k=1,2,\dotsc,5$, \resp.
%\end{lemma}

The stability results for BDF methods are the following.
\begin{lemma}
\label{lemma: BDF stability}
    For a $k$-step implicit or linearly implicit BDF method with $k\leq5$ there exists a $\tau_0 >0$, %depending only on the constants $\mu$ and $\kappa$,
    \st\ for $\tau\leq\tau_0$ and $t_n=n \tau\leq T$, that the error $e_n$ is bounded by
    \begin{equation*}
        \disp |e_n|_{\BM_n}^2 + \tau \sum_{j=k}^n |e_j|_{\BA_j}^2 \leq C \tau \sum_{j=k}^n \|d_j\|_{\ast, t_j}^2 + C \max_{0\leq i \leq k-1} |e_i|_{\BM_i}^2
    \end{equation*}
    where $\|w\|_{\ast, t}^2=w^T(\BA\t+M\t)\inv w$. The constant $C$ is
    independent of $h, \tau$ and $n$ (but depends on $\m$, $\M$, $L$, $\mu$, $\kappa$ and $T$).
\end{lemma}
\begin{proof}
The proof follows the proof of Lemma 4.1 from \cite{LubichMansourVenkataraman_bdsurf} (using $G$-stability from \cite{Dahlquist} and multiplier techniques from \cite{NevanlinnaOdeh}), except in those terms where the nonlinearity appears. For their estimates we refer to Theorem~1 in \cite{AkrivisLubich_quasilinBDF}. For linearly implicit methods we follow \cite[Section~6]{AkrivisLubich_quasilinBDF}. Therefore these proofs are also omitted.
\end{proof}
%
%\begin{remark}
%    We not here that our stability conditions, also appeared in \cite{AkrivisLubich_quasilinBDF} see equations (23), (42) and (48) therein, is very mild.
%
%    If $\m > \eta_5=0.8160$ then the condition is immediately fulfilled. However, if the \op\ $A(u)$ is \emph{hermitian} and positive definite then one can set
%    \begin{equation*}
%        \|w\|_{t,u}^2 := \la \A(u)w,w \ra_t,
%    \end{equation*}
%    hence $\m=1$ holds. Using the techniques from \cite[Section 7]{AkrivisLubich_quasilinBDF} and introducing the above space-dependent norm such restrictions are avoided, see the proof of Theorem 4 therein.
%\end{remark}

Again, using the above stability results, the error bounds are following analogously as in \cite[Theorem~5.1]{LubichMansourVenkataraman_bdsurf} (or \cite[Theorem~5.3]{diss_Mansour}).
\begin{theorem}
\label{thm: BDF error estimates}
     Consider the quasilinear parabolic problem \eqref{eq: strong form}, having a solution in $\S\t$ for $0\leq t \leq T$. Couple the evolving surface finite element method as space discretization with time discretization by a $k$-step implicit or linearly implicit backward difference formula of order $k\leq5$. Assume that the Ritz map of the solution has \co\ discrete material derivatives up to order $k+1$. Then there exists $\tau_0>0$, independent of $h$, \st\ for $\tau \leq \tau_0$, for the error $E_h^n=U_h^n-\P u(.,t_n)$ the following estimate holds for $t_n=n\tau \leq T$:
    {\setlength\arraycolsep{.13889em}
    \begin{eqnarray*}
        \disp \|E_h^n\|_{L^2(\Ga_h(t_n))} &+& \Big( \tau \sum_{j=1}^n \|\nb_{\Ga_h(t_j)} E_h^j \|_{L^2(\Ga_h(t_j))}^2 \Big)^{\frac{1}{2}} \\
        \disp \leq C \tilde{\beta}_{h,k} \tau^{k} &+& \Big( \tau \sum_{j=1}^n \|R_h(.,t_j) \|_{H^{-1}(\Ga_h(t_j))}^2 \Big)^{\frac{1}{2}} + C \max_{0\leq i \leq k-1} \|E_h^i\|_{L^2(\Ga_h(t_i))},
    \end{eqnarray*}}
    where the constant $C$ is independent of $h, n$ and $\tau$ (but depends on $\m$, $\M$, $L$, $\mu$, $\kappa$ and $T$). Furthermore
    \begin{equation*}
        \disp \tilde{\beta}_{h,k}^2 = \int_0^T \sum_{\ell=1}^{k+1} \| (\mat_h)^{(\ell)} (\P u)(.,t) \|_{L^2(\Ga_h(t))} \d t.
    \end{equation*}
\end{theorem}

\section{Error bounds for the fully discrete solutions}
\label{section: error bounds}

We follow the approach of \cite[Section~5]{LubichMansourVenkataraman_bdsurf} by defining the FEM residual $R_h(.,t) = \sum_{j=1}^N r_j\t \chi_j(.,t)\in S_h\t$ as
%\begin{equation}\label{eq: residual}
%    \int_{\Ga_h}\!\!\!\!  R_h \phi_h = \diff \int_{\Ga_h}\!\!\!\!  (\P u) \phi_h +\int_{\Ga_h}\! \sum_{i,j=1}^{m+1} a_{ij}(u) \big(\nbg (\P u)\big)_i \big(\nbg \phi_h\big)_j - \int_{\Ga_h}\!\!\!\! (\P u)  \mat_h \phi_h,
%\end{equation}
\begin{equation}\label{eq: residual}
    \int_{\Ga_h}\!\!\!\!  R_h \phi_h = \diff \int_{\Ga_h}\!\!\!\!  \P u \phi_h +\int_{\Ga_h}\! \A(\P u) \nbg (\P u) \cdot \nbg \phi_h - \int_{\Ga_h}\!\!\!\! (\P u)  \mat_h \phi_h,
\end{equation}
where $\phi_h\in S_h\t$, and the Ritz map of the true solution $u$ is given as
\begin{equation*}
    \Pt u(.,t) = \sum_{j=1}^N \widetilde{\alpha}_j\t \chi_j(.,t).
\end{equation*}
The above problem is equivalent to the ODE system with the vector $r\t=(r_j\t)\in\R^N$:
\begin{equation*}
    \diff \big(M\t \widetilde{\alpha}\t\big) + A(\widetilde{\alpha}\t) \widetilde{\alpha}\t = M\t r\t,
\end{equation*}
which is the perturbed ODE system \eqref{eq: pertubated ODE}.

\subsection{Bound of the semidiscrete residual}

We now show the optimal second order estimate of the residual $R_h$.

\begin{theorem}
\label{thm: res bound}
    Let $u$, the solution of the parabolic problem, be in $\S\t$ for $0\leq t \leq T$. Then there exists a constant $C>0$ and $h_0>0$, \st\ for all $h\leq h_0$ and $t\in[0,T]$, the finite element residual $R_h$ of the Ritz map is bounded as
    \begin{equation*}
        \|R_h\|_{H\inv(\Ga_h\t)} \leq C h^2.
    \end{equation*}
\end{theorem}

\begin{proof}
(a) We start by applying the discrete transport property to the residual equation \eqref{eq: residual}
{\setlength\arraycolsep{.13889em}
\begin{eqnarray*}
    m_h(R_h,\phi_h) &=& \diff m_h(\Pt u,\phi_h) + a_h(\Pt u;\Pt u,\phi_h) - m_h(\Pt u,\mat_h\phi_h) \\
    &=& m_h(\mat_h\Pt u,\phi_h) + a_h(\Pt u;\Pt u,\phi_h) + g_h(V_h;\Pt u,\phi_h).
\end{eqnarray*}}

(b) We continue by the transport property with discrete material derivatives from Lemma \ref{lemma: transport prop}, but for the weak form, with $\vphi:=\vphi_h=(\phi_h)^l$:
{\setlength\arraycolsep{.13889em}
\begin{eqnarray*}
    0 &=& \diff m(u,\vphi_h) + a(u;u,\vphi_h) - m(u,\mat\vphi_{h}) \\
    &=& m(\mat_h u,\vphi_h) + a(u;u,\vphi_h) + g(v_h;u,\vphi_h) +
    m(u,\mat_{h}\varphi_{h} - \mat \varphi_{h}).
\end{eqnarray*}}

(c) Subtraction of the two equations, using the definition of the Ritz map with $\xi=u$
in \eqref{eq: Ritz definition}, i.e.\ $$a_{h}^{*}( u^{-l}; \Pt u, \phi_{h}) = a^{*}(u; u, \varphi_{h}),$$ and using that
\[
    \mat_{h}\varphi_{h} - \mat \varphi_{h} = (v_{h}-v) \cdot \nabla_{\Gamma}\varphi_{h}
\]
holds, we obtain
{\setlength\arraycolsep{.13889em}
\begin{eqnarray*}
    m_h(R_h,\phi_h) &=& m_h(\mat_h\Pt u,\phi_h) - m(\mat_h u,\vphi_h) \\
                    &+& g_h(V_h;\Pt u,\phi_h) - g(v_h;u,\vphi_h) \\
                    &+& a^\ast_{h}(\Pt u;  \Pt u ,\phi_h) - a^\ast_{h}(u^{-l}; \Pt u,\phi_h) \\
                    &+& m(u,\vphi_h) - m_h(\Pt u, \phi_h) \\
                    &+& m(u, (v_{h}-v)  \cdot \nabla_{\Gamma}\varphi_{h}).
\end{eqnarray*}}
All the pairs can be easily estimated separately as $c h^2 \|\vphi_h\|_{L^2(\Gat)}$, by combining the estimates of Lemma \ref{lemma: estimation of forms}, and Theorem \ref{thm: Ritz error} and \ref{thm: Ritz mat error}, except the third, and the last term.

The term containing the velocity difference $(v_h-v)$ can be estimated, using $|v_h-v| + h |\nbg (v_h-v)| \leq ch^2$ from \cite[Lemma~5.6]{DziukElliott_L2}, as $c h^2 \|\nbg \vphi_h\|_{L^2(\Gat)}$.

The nonlinear terms are rewritten as:
\begin{align*}
    a_{h}^{\ast}(\Pt u;  \Pt u ,\phi_h) - a^\ast_{h}(u^{-l}; \Pt u,\phi_h)  & =
    a_{h}^{\ast}(\Pt u; \Pt u, \phi_{h} ) - a^{\ast}(\mathcal{P}_{h}u; \mathcal{P}_{h}u, \varphi_{h})\\
    & + a^{\ast}(\mathcal{P}_{h}u; \mathcal{P}_{h}u, \varphi_{h}) -
    a^{*}(u; \mathcal{P}_{h}u, \varphi_{h}) \\
    & + a^{\ast}(u; \mathcal{P}_{h}u, \varphi_{h}) - a_{h}^{\ast}(u^{-l}; \Pt u, \phi_{h})
\end{align*}
For the first and the third term Lemma~\ref{lemma: estimation of forms} provides an upper bound $c h^{2} \|\nbg \vphi_h\|_{L^2(\Gat)}$ (similarly like before).

Finally, using Lemma~\ref{lemma: Ritz regularity lemma} we obtain, similarly to \eqref{Lipschitz}, that the second term can be bounded as
\begin{align*}
    &\big|a^\ast(\mathcal{P}_{h} u; \mathcal{P}_{h} u ,\phi_h) - a^\ast(u; \mathcal{P}_{h} u ,\phi_h)\big| \\
    &= \Big| \int_{\Ga\t} \big(\A(\mathcal{P}_{h} u) - \A(u)\big) \nbg \mathcal{P}_{h} u \cdot \nbg \vphi_h \Big| \\
    &\leq c \ell \|\mathcal{P}_{h} u - u\|_{L^2(\Ga\t)} \|\nbg \mathcal{P}_{h} u\|_{L^\infty(\Ga\t)} \|\nbg \vphi_h\|_{L^2(\Ga\t)} \\
    &\leq c \ell \|\mathcal{P}_{h} u - u\|_{L^2(\Ga\t)} \ c \ r \ \|\nbg \vphi_h\|_{L^2(\Ga\t)} \\
    &\leq c \ell r\, h^{2} \|\nbg \vphi_h\|_{L^2(\Gat)}.
\end{align*}
\end{proof}

\subsection{Error estimates for the full discretizations}

We compare the lifted fully discrete numerical solution $u_h^n:=(U_h^n)^l$ with the exact solution $u(.,t_n)$ of the evolving surface PDE \eqref{eq: strong form}, where $U_h^n = \sum_{j=1}^N \alpha_j^n\chi_j(.,t)$, where the vectors $\alpha^n$ are generated by a Runge--Kutta or a BDF method.

\begin{theorem}[ESFEM and R--K]\label{thm: ESFEM RK error}
    Consider the evolving surface finite element method as space discretization of the quasilinear parabolic problem \eqref{eq: strong form}, with time discretization by an $s$-stage implicit Runge--Kutta method satisfying Assumption \ref{assump: RK method assumptions}. Let $u$ be a sufficiently smooth solution of the problem, which satisfies $u(.,t) \in \S\t$ ($0\leq t \leq T$), and assume that the initial value is approximated as
    \begin{equation*}
        \disp \|u_h^0 - (\P u)(.,0)\|_{L^2(\Ga(0))} \leq C_0 h^2.
    \end{equation*}
    Then there exists $h_0>0$ and $\tau_0>0$, \st\ for $h\leq h_0$ and $\tau \leq \tau_0$, the following error estimate holds for $t_n=n\tau \leq T$:
    \begin{equation*}
        \|u_h^n - u(.,t_n)\|_{L^2(\Ga(t_n))} + h\Big( \! \tau \!\! \sum_{j=1}^n \|\nb_{\Ga(t_j)} (u_h^j - u(.,t_j))\|_{L^2(\Ga(t_j))}^2 \! \Big)^{\frac{1}{2}} \!\! \leq  C \big( \tau^{q+1} \! + h^2 \big).
    \end{equation*}
    The constant $C$ is independent of $h, \ \tau$ and $n$, but depends on $\m, \ \M, \ L, \ \mu, \ \kappa$ and $T$.
\end{theorem}

\begin{theorem}[ESFEM and BDF]\label{thm: ESFEM BDF error}
    Consider the evolving surface finite element method as space discretization of the quasilinear parabolic problem \eqref{eq: strong form}, with time discretization by a $k$-step implicit or linearly implicit backward difference formula of order $k\leq5$. Let $u$ be a sufficiently smooth solution of the problem, which satisfies $u(.,t) \in \S\t$ ($0\leq t \leq T$), and assume that the starting values are satisfying
    \begin{equation*}
        \disp \max_{0 \leq i \leq k-1} \| u_h^i - (\P u)(.,t_i) \|_{L^2(\Ga(0))} \leq C_0 h^2.
    \end{equation*}
    Then there exists $h_0>0$ and $\tau_0>0$, \st\ for $h\leq h_0$ and $\tau \leq \tau_0$, the following error estimate holds for $t_n=n\tau \leq T$:
    \begin{equation*}
        \disp \|u_h^n - u(.,t_n)\|_{L^2(\Ga(t_n))} + h\Big( \! \tau \!\! \sum_{j=1}^n \|\nb_{\Ga(t_j)} (u_h^j - u(.,t_j))\|_{L^2(\Ga(t_j))}^2 \! \Big)^{\frac{1}{2}} \leq C \big( \tau^{k} \! + h^2 \big).
    \end{equation*}
    The constant $C$ is independent of $h, \ \tau$ and $n$, but depends on $\m, \ \M, \ L, \ \mu, \ \kappa$ and $T$.
\end{theorem}

\begin{proof}[Proof of Theorem~\ref{thm: ESFEM RK error}--\ref{thm: ESFEM BDF error}]
    The global error is decomposed into two parts:
    \begin{equation*}
        \disp u_h^n - u(.,t_n) = \Big(u_h^n - (\P u)(.,t_n)\Big) + \Big((\P u)(.,t_n) - u(.,t_n)\Big),
    \end{equation*}
    and the terms are estimated by previous results.

    The first one is estimated by our results for Runge--Kutta or BDF methods: Theorem \ref{thm: RK error estimates} or \ref{thm: BDF error estimates}, \resp, together with the residual bound Theorem \ref{thm: res bound}, and by the Ritz error estimates Theorem~\ref{thm: Ritz error} and \ref{thm: Ritz mat error}.

    The second term is estimated by the error estimates for the Ritz map (Theorem \ref{thm: Ritz error} and \ref{thm: Ritz mat error}).
\end{proof}

\section{Further extensions}
\label{section: extensions}

\subsubsection*{Semilinear problems}

The presented results, in particular Theorem~\ref{thm: ESFEM RK error} and \ref{thm: ESFEM BDF error}, can be generalized to semilinear problems. Convergence results for BDF method were already shown for semilinear problems in \cite{AkrivisLubich_quasilinBDF}. For the analogous results for Runge--Kutta methods follow \cite[Remark~1.1]{LubichOstermann_RK}. Problems fitting into this framework can be found in the references given in the introduction.

The inhomogeneity $f\t$ in the evolving surface PDE \eqref{eq: strong form} can be replaced by $f(t,u)$ satisfying a local Lipschitz condition (similar to \eqref{Lipschitz}): for every $\delta>0$ there exists $L=L(\delta,r)$ \st\
\begin{equation*}
    \| f(t,w_1) - f(t,w_2)\|_{V\t'} \leq \delta\|w_1 - w_2\|_{V\t} + L \|w_1 - w_2\|_{H\t} \quad (0\leq t \leq T)
\end{equation*}
holds for arbitrary $w_1,w_2 \in V\t$ with $\|w_1\|_{V\t},\,\|w_2\|_{V\t} \leq r$, uniformly in $t$. Such a condition can be satisfied by using the same $\S$ set as for quasilinear problems.

To be precise: In this case the bilinear form $a(.,.)$ is not depending on $\xi$, it is as in \cite{DziukElliott_L2}. Section \ref{section: ESFEM} would reduce to recall results mainly from \cite{DziukElliott_ESFEM,DziukElliott_L2}. The stability estimates for the Runge--Kutta and BDF methods are needed to be revised in a straightforward way, cf.\ \cite{LubichOstermann_RK} and \cite{AkrivisLubich_quasilinBDF}, \resp. The generalized Ritz map is the one appeared in \cite{LubichMansour_wave,diss_Mansour} together with its error bounds. The regularity result of the Ritz map still needed from Section~\ref{section: Ritz}.

\subsubsection*{Deteriorating constants}

In view of the cited papers, especially \cite[Remark~1.1]{LubichOstermann_RK}, Theorem~\ref{thm: ESFEM RK error} and \ref{thm: ESFEM BDF error} have an extension to the situation where the constants $\m$ and $\M$, in \eqref{ellipticity} and \eqref{boundedness} are depending on $\|u\|$, and allowed to deteriorate as $\|u\|$ tends to infinity. Using energy estimates deteriorating constants can be handled for nonlinear problems. Then the constant $C$ in Theorem~\ref{thm: ESFEM RK error} and \ref{thm: ESFEM BDF error} depends also on $\sup_{t\in[0,T]} \|u\t\|$.
For instance the incompressible Navier--Stokes equation are fitting into this framework.

\section{Numerical experiments}
\label{section: numerics}

We present a numerical experiment for an evolving surface quasilinear parabolic problem discretized by  evolving surface finite elements coupled with the backward Euler method as a time integrator. The fully discrete methods were implemented in DUNE-FEM \cite{dunefem}, while the initial triangulations were generated using DistMesh \cite{distmesh}.

The evolving surface is given by
\begin{equation*}
    \Ga\t = \big\{ x\in\R^3 \ \big| \ a(t)\inv x_1^2+ x_2^2 +x_3^2 -1 = 0 \big\},
\end{equation*}
where $a(t)=1+0.25\sin(2\pi t)$, see e.g.\ \cite{DziukElliott_ESFEM,DziukLubichMansour_rksurf,diss_Mansour}. The problem is considered over the time interval $[0,1]$. We consider the problem with the nonlinearity $\A(u)=1-\frac{1}{2}e^{-x^2/4}$. The right-hand side $f$ is computed as to have $u(x,t)=e^{-6t} x_1 x_2$ as the true solution of the quasilinear problem
\begin{equation*}
    \begin{cases}
    \begin{alignedat}{4}
        \mat u + u \nb_{\Gat} \cdot v - \nb_{\Gat} \cdot \Big( \A(u) \nb_{\Gat} u\Big) &= f & \qquad & \textrm{ on } \Ga\t ,\\
        u(.,0) &= u_0 & \qquad & \textrm{ on } \Ga(0).
    \end{alignedat}
    \end{cases}
\end{equation*}

\medskip
Let $(\mathcal{T}_k\mathcal{}(t))_{k=1,2,\dotsc,n}$ and $(\tau_k)_{k=1,2,\dotsc,n}$ be a series of triangulations and timesteps, \resp, such that $2 h_k\approx h_{k-1}$ and $4 \tau_k = \tau_{k-1}$, with $\tau_1=0.1$. By $e_k$ we denote the error corresponding to the mesh $\mathcal{T}_k\t$ and stepsize $\tau_k$. Then the EOCs are given as
\begin{equation*}
    EOC_{k}=\frac{\ln(e_{k}/e_{k-1})}{\ln(2)}, \qquad (k=2,3, \dotsc,n).
\end{equation*}
In Table~\ref{table: EOCs} we report on the EOCs, for the ESFEM coupled with backward Euler method, corresponding to the norms
\begin{alignat*}{3}
    L^\infty(L^2):&\qquad \max_{1\leq n \leq N}\|u_h^n - u(.,t_n)\|_{L^2(\Ga(t_n))},\\
    L^2(H^1):&\qquad  \Big(\tau \sum_{n=1}^{N} \|\nb_{\Ga(t_n)}\big(u_h^n - u(.,t_n)\big)\|_{L^2(\Ga(t_n))}\Big)^{1/2}.
\end{alignat*}
\begin{table}[!ht]
    \centering
    \begin{tabular}{r r  l l l l}
        \toprule
        level & dof & $L^\infty(L^2)$ & EOCs & $L^2(H^1)$ & EOCs \\
        \midrule
            1 & 126   & 0.07121892 & -    & 0.1404349 & - \\
            2 & 516   & 0.02077452 & 1.78 & 0.0404614 & 1.80 \\
            3 & 2070  & 0.00540906 & 1.94 & 0.0111377 & 1.86 \\
            4 & 8208  & 0.00136755 & 1.98 & 0.0033538 & 1.73 \\
            5 & 32682 & 0.00034289 & 2.00 & 0.0011904 & 1.49 \\
%            1 & 126   & 7.121892e-02 & - & 1.404349e-01 & - \\
%            2 & 516   & 2.077452e-02 & 1.78 & 4.046140e-02 & 1.80 \\
%            3 & 2070  & 5.409067e-03 & 1.94 & 1.113774e-02 & 1.86 \\
%            4 & 8208  & 1.367556e-03 & 1.98 & 3.353860e-03 & 1.73 \\
%            5 & 32682 & 3.428909e-04 & 2.00 & 1.190425e-03 & 1.49 \\
        \bottomrule
    \end{tabular}
    \caption{Errors and EOCs in the $L^\infty(L^2)$ and $L^2(H^1)$ norms}
    \label{table: EOCs}
\end{table}

Figure~\ref{fig: quasilin_ESFEM_IE} shows the errors obtained by the backward Euler method coupled with ESFEM for four different meshes and a series of time steps. The convergence in time can be seen (note the reference line), while for sufficiently small $\tau$ the spatial error is dominating, in agreement with the theoretical results.
\pagebreak 
\begin{figure}[!ht]
  \centering
  \includegraphics[width=\textwidth]{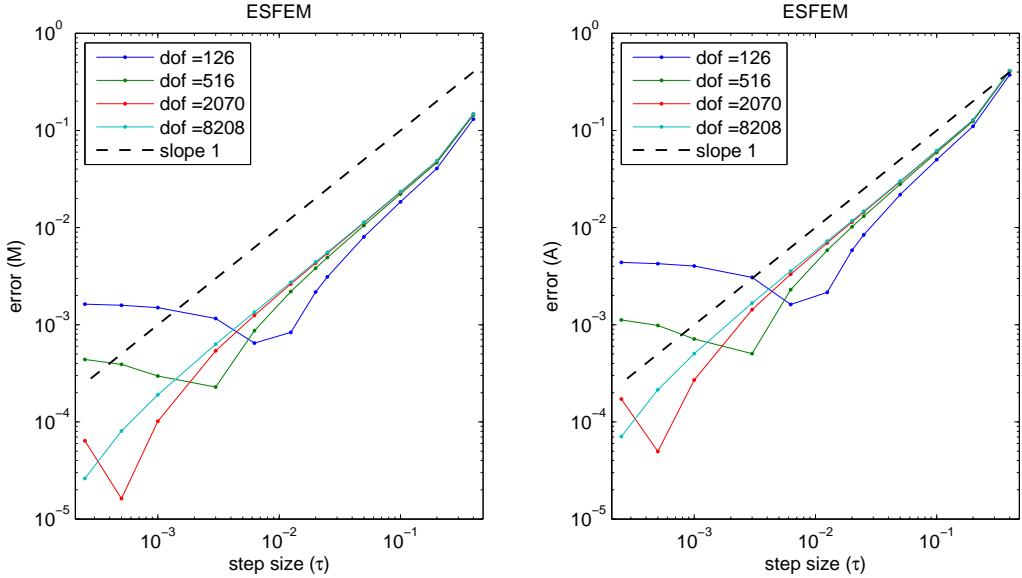}\\
  \caption{$\|.\|_{\BM}$-errors of the ESFEM and the backward Euler method at time $T=1$}
  \label{fig: quasilin_ESFEM_IE}
\end{figure}

Figure~\ref{fig: BDF3_LI} shows the errors obtained by the three step linearly implicit BDF method coupled with ESFEM for five different meshes and a series of time steps. Again the results are matching with the theoretical ones.
\begin{figure}[!ht]
  \centering
  \includegraphics[width=\textwidth]{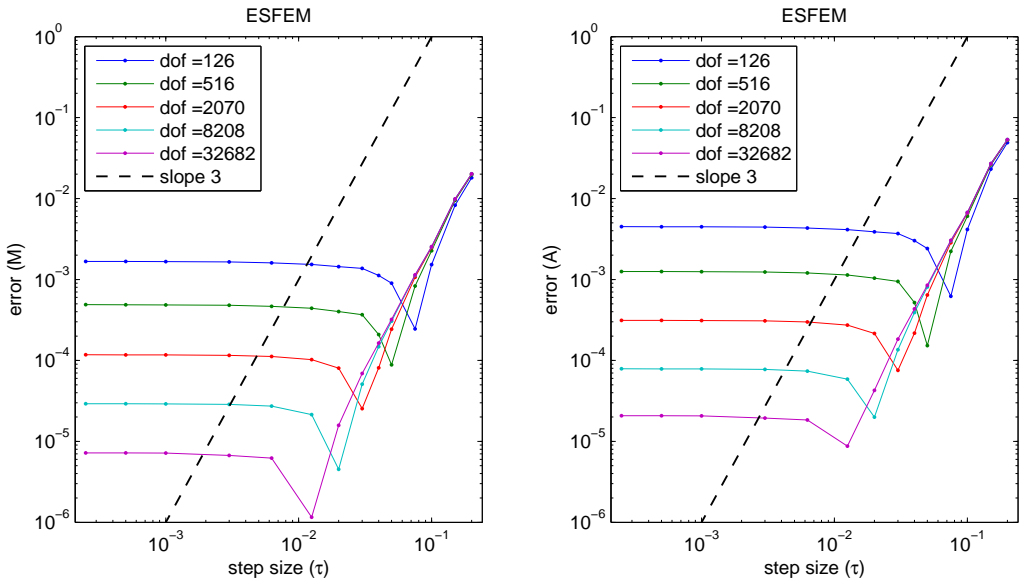}\\
  \caption{$\|.\|_{\BA}$-errors of the ESFEM and the $3$ step linearly implicit BDF method at time $T=1$}
  \label{fig: BDF3_LI}
\end{figure}

We note that, for this example, no significant difference appeared between the fully implicit and linearly implicit BDF methods.

\section*{Acknowledgement}

The authors would like to thank Prof.\ Christian Lubich for the invaluable discussions on the topic and for his encouragement and help during the preparation of this paper.
We would also like to thank Prof.\ Frank Loose and Christopher Nerz for our discussions on the topic.
The research stay of B.K.\ at the University of T\"{u}bingen has been funded by the Deutscher Akademischer Austausch Dienst (DAAD).

\bibliographystyle{alpha}
\bibliography{nonlin_literature}

\appendix

\section{A priori estimates}
\label{sec:cald-zygm-relat}

The result presented here gives regularity result, with a $t$ independent constant, for the elliptic problems appeared in the proofs of the errors in the Ritz map.
\begin{theorem}[Elliptic regularity for evolving surfaces]
  \label{lemma:EllipticRegularity}
  Let $\surface(t)$ be an evolving surface, fix a $t\in
  [0,T]$ and a function $\xi\colon \surface(t)\to \R$.
  \begin{enumerate}[label=(\roman*)]
  \item Let $f\in H^{-1}\bigl(\surface(t)\bigr)$ and
    \begin{align}
      \label{eq:ell_Op}
      L(u) :=  -\nbg \cdot \big(\A(\xi) \nbg u \big) + u.
    \end{align}
    Then there exists a weak solution $u\in H^{1}\bigl(\surface(t)\bigr)$ of the
    problem
    \begin{align}
      \label{eq:PDE_Problem}
      L(u) = f
    \end{align}
    with the estimate
    \begin{align}
      \label{eq:minor_reg_estimate}
      \Hnorm[\surface(t)]{u} \leq c \Hnorm[\surface(t)][-1]{f},
    \end{align}
    where the constant above is independent of $t$.
  \item Let $L(u)$ be \eqref{eq:ell_Op}, let $f\in L^{2}\bigl(\surface(t)\bigr)$ and
    let $u\in H^{1}\bigl(\surface(t)\bigr)$ be a weak solution of
    \eqref{eq:PDE_Problem}.  Then $u$ is a strong solution of
    \eqref{eq:PDE_Problem}, i.e.\ $u$ solves \eqref{eq:PDE_Problem} almost
    everywhere and there exists a constant $c>0$ independent of $t$ and $u$ such
    that
    \[
    \Hnorm[\surface(t)][2]{u} \leq c \bigl(\Lnorm[\surface(t)]{u} +
    \Lnorm[\surface(t)]{f} \bigr).
    \]
  \end{enumerate}
\end{theorem}

\begin{proof}
  For (i): The Lax--Milgram lemma shows the existence of the weak solution $u$.
  Because the coercivity and boundedness constants \eqref{ellipticity} and
  \eqref{boundedness} are independent of $t$, the constant in
  \eqref{eq:minor_reg_estimate} also not depends on $t$. For (ii): Basically we consider pullback of the operator $L$ to $\Ga(0)$, rewrite it in a local chart and then apply the corresponding results of \cite{GilbargTrudinger}.

  By assumption there exists a diffeomorphic parametrization of our evolving
  surface $\surface(t)$, i.e.\ we have a smooth map
  \[
  \Phi\colon \surface(0) \times [0,T] \to \R^{m+1}
  \]
  such that
  \[
  \Phi_{t} \colon \surface(0) \to \R^{m+1}, \quad \Phi_{t}(x) := \Phi(x,t)
  \]
  is an injective immersion which is a homeomorphism onto its image with
  $\Phi_{t}\bigl(\surface(0)\bigr) = \surface(t)$.  Because $\surface(0)$ is
  compact, there exists a finite atlas
  \[
  \Bigl(\varphi_{n}(0)\colon U_{n}(0)\subset \surface(0)\to
  \R^{m}\Bigr)_{n=1}^{k}
  \]
  such that $\varphi_{n}\bigl(U_{n}(0)\bigr)\subset \R^{m}$ is bounded and a
  finite family of compact sets $\bigl(V_{n}(0)\bigr)_{n=1}^{k}$ with
  $V_{n}(0)\subset U_{n}(0)$, and $\bigcup_{n=1}^{k}V_{n}(0) = \surface(0)$.
  Using the properties of the diffeomorphic parametrization the new
  collections,
  \[
  V_{n}(t):= \Phi_{t}\bigl(V_{n}(0)\bigr),\quad U_{n}(t):=
  \Phi_{t}\bigl(U_{n}(0)\bigr),\quad \varphi_{n}(t) :=
  \varphi_{n}(0)\circ \Phi_{t}^{-1},
  \]
  still have the same properties. %  \par
  Now consider the following standard formulae of Riemannian geometry
  \cite{MCFEckart}:
  \[
  \nbg h (x,t) = \sum_{i,j = 1}^{m} g^{ij}_{n}(x,t) \frac{\partial( h \circ
    \varphi_{n}(t)^{-1})}{\partial x^{i}} \frac{\partial
    \bigl(\varphi_{n}(t)^{-1} \bigr)}{\partial x^{j}},
  \]
  where
  \[
  g_{ij,n}(x,t):= \left. \frac{\partial \bigl(\varphi_{n}(t)^{-1}
      \bigr)}{\partial x^{i}} \cdot \frac{\partial \bigl(\varphi_{n}(t)^{-1}
      \bigr)}{\partial x^{j}}\right\rvert_{x}
  \]
  is the first fundamental form and $g^{ij}_{n}(x,t)$ are entries of the
  inverse matrix of $g_{n} := (g_{ij,n})$, and
  \[
  \nbg \cdot X = \sum_{i,j=1}^{m}\frac{1}{\sqrt{g_{n}}}
  \frac{\partial}{\partial x^{i}} \bigl(\sqrt{g_{n}} g^{ij}_{n}
  X_{j}\bigr)
  \]
  where $X$ is a smooth tangent vector field with $X_{j} = X \cdot
  \frac{\partial \bigl( \varphi(t)^{-1} \bigr)}{\partial x^{j}}$ and $\sqrt{g_{n}}:=
  \sqrt{\det(g_{n})}$.  It is straightforward to calculate that
  \begin{align*}
    \Bigl(- \nbg \cdot \A \nbg u + u \Bigr)\circ \varphi_{n}(t)^{-1}(x) &=
    \sum_{i,j=1}^{m} a_{ij,n}(x,t) \frac{\partial^{2}\bigl( u \circ
      \varphi_{n}(t)^{-1}\bigr) }{\partial x^{i} \partial x^{j}} +
    \sum_{i=1}^{m} b_{i,n}(x,t) \frac{\partial
      \bigl( u\circ \varphi_{n}(t)^{-1}\bigr)}{\partial x^{i}} \\
    & \hphantom{==} + c_{n}(x,t) \, u\circ \varphi_{n}(t)^{-1}
  \end{align*}
  for some appropriate functions $a_{ij,n}\in W^{1,\infty}\bigl(U_{n}(t)\bigr)$,
  $b_{i,n},c_{n}\in L^{\infty}\bigl(U_{n}(t)\bigr) $ where $a_{ij,n}$ represents
  a uniform elliptic matrix.  Observe that the assumptions \eqref{ellipticity},
  \eqref{boundedness} and \eqref{Lipschitz} implies that the function above can
  be bounded independently of $t$. Now \cite[Theorem~8.8]{GilbargTrudinger} states that, if $u\circ
  \varphi_{n}(t)^{-1}$ is the $H^{1}$-weak solution of \eqref{eq:PDE_Problem}, then
  it must be a strong solution as well.

  For the estimate in (ii) observe that
  \cite[Theorem~9.11]{GilbargTrudinger} gives us for $V_{n}(t)$ in particular the estimate
  \begin{align}
    \label{eq:preRegEstimate}
    \Hnorm[V_{n}'(t)][2]{u\circ \varphi_{n}(t)^{-1}} \leq c (
    \Lnorm[U_{n}'(t)]{u\circ \varphi_{n}(t)^{-1}} +
    \Lnorm[U_{n}'(t)]{f\circ \varphi_{n}(t)^{-1}}),
    % \Wnorm[V_{n}'(t)][2,p]{u\circ \varphi_{n}(t)^{-1}} \leq c (
    % \Lnorm[U_{n}'(t)][p]{u\circ \varphi_{n}(t)^{-1}} +
    % \Lnorm[U_{n}'(t)][p]{f\circ \varphi_{n}(t)^{-1}}),
  \end{align}
  where $V_{n}' := \varphi_{n}(t)\bigl(V_{n}(t)\bigr)$ and
  $U_{n}' := \varphi_{n}(t)\bigl(U_{n}(t)\bigr)$ are obviously
  independent of $t$.  Thus the constant above is independent of $t$.  Then Theorem~3.41 in \cite{Adams} shows that
  \[
  \Hnorm[V_{n}(t)][2]{u } \leq c(t) \Hnorm[V_{n}'(t)][2]{u\circ
    \varphi_{n}(t)^{-1}} \leq c \Hnorm[V_{n}'(t)][2]{u\circ
    \varphi_{n}(t)^{-1}},
  % \Wnorm[V_{n}(t)][2,p]{u } \leq c(t) \Wnorm[V_{n}'(t)][2,p]{u\circ
  %   \varphi_{n}(t)^{-1}} \leq c \Wnorm[V_{n}'(t)][2,p]{u\circ
  %   \varphi_{n}(t)^{-1}}
  \]
  where the constant in the middle depends continuously on $t$, hence the
  last constant is independent of $t$. A similar estimate holds for the right-hand side of \eqref{eq:preRegEstimate}.  An easy calculation finishes the
  proof for (ii).
\end{proof}

\end{document}